\documentclass[12pt]{amsproc}
\usepackage{amssymb}
\usepackage{calrsfs}
\newcommand{\Z}{{\mathbb{Z}}}
\newcommand{\Q}{{\mathbb{Q}}}

\newcommand{\R}{{\mathbb{R}}}
\newcommand{\C}{{\mathbb{C}}}
\newcommand{\gl}{{\mathfrak{gl}}}

\newcommand{\fg}{{\mathfrak{g}}}
\newcommand{\fh}{{\mathfrak{h}}}

\newcommand{\be}{{\mathbf{e}}}
\newcommand{\bbf}{{\mathbf{f}}}
\newcommand{\cO}{{\mathcal{O}}}

\newcommand{\dalp}{{\dot{\alpha}}}
\newcommand{\dmu}{{\dot{\mu}}}

\newcommand{\dPhi}{{\dot{\Phi}}}

\newcommand{\Hom}{{\operatorname{Hom}}}
\newcommand{\GL}{{\operatorname{GL}}}
\renewcommand{\leq}{\leqslant}
\renewcommand{\geq}{\geqslant}

\newtheorem{thm}{Theorem}[section]
\newtheorem{lem}[thm]{Lemma}
\newtheorem{prop}[thm]{Proposition}
\newtheorem{cor}[thm]{Corollary}

\theoremstyle{definition}
\newtheorem{defn}[thm]{Definition}
\newtheorem{exmp}[thm]{Example}

\theoremstyle{remark}
\newtheorem{rem}[thm]{Remark}

\numberwithin{equation}{section}

\begin{document}
\title{Minuscule weights and Chevalley groups}
\author{Meinolf Geck}
\address{IAZ - Lehrstuhl f\"ur Algebra\\Universit\"at Stuttgart\\
Pfaffenwaldring 57\\D--70569 Stuttgart\\ Germany}
\curraddr{}
\email{meinolf.geck@mathematik.uni-stuttgart.de}
\thanks{This work was supported by DFG Priority Programme SPP 1489.}


\date{}

\begin{abstract}
The traditional construction of Chevalley groups relies on the choice
of certain signs for a Chevalley basis of the underlying Lie 
algebra~$\fg$. Recently, Lusztig simplified this construction for 
groups of adjoint type by using the ``canonical basis'' of the adjoint 
representation of~$\fg$; in particular, no choices of signs are required. 
The purpose of this note is to extend this to Chevalley groups which 
are not necessarily of adjoint type, using Jantzen's explicit
models of the minuscule highest weight representations of~$\fg$.
\end{abstract}

\maketitle

\section{Introduction} \label{sec0}

Let $\fg$ be a finite-dimensional semisimple Lie algebra over $\C$.
Chevalley \cite{Ch} found a uniform procedure by which one can
associate with any field $K$ a corresponding group $G_K$ of ``type $\fg$''.
If $K$ is algebraically closed, then $G_K$ is a semisimple algebraic group
of adjoint type with root system equal to that of~$\fg$; if $K$ is finite, 
then this construction led to some new families of finite simple groups.
A detailed exposition can be found in the book by Carter \cite{Ca1}. The 
subsequent work of Chevalley \cite{Ch2} leads to an extension of the theory
by which one also obtains groups which are not necessarily of adjoint type; 
this, and much more, is contained in Steinberg's lecture notes \cite{St}.
As a typical example, if $\fg=\mathfrak{sl}_n(\C)$, then 
Chevalley's original construction yields the group $G_K=\mbox{PGL}_n(K)$. In 
order to obtain $\mbox{SL}_n(K)$ and all intermediate cases between 
$\mbox{SL}_n(K)$ and $\mbox{PGL}_n(K)$, one has to work in the setting 
of~\cite{St} which requires a bit more Lie theory,
most notably the basics about highest weight theory and Kostant's $\Z$-form 
of the universal enveloping algebra of~$\fg$. More recently, Lusztig
\cite{L3} used the theory of ``canonical bases'' (in the sense of 
\cite{L6}) to give a uniform procedure which even yields reductive 
algebraic groups of any type. 

For various applications, it is useful to be able to write down explicit
matrix representations for $\fg$ and $G_K$; such applications include, 
for example:
\begin{itemize}
\item The computation of nilpotent orbits and unipotent classes 
(especially in bad characteristic), see Holt--Spaltenstein \cite{HoSp}
and further references there.
\item The determination of composition multiplicities in Weyl modules for 
finite Chevalley groups, see L\"ubeck \cite{Lue} and further references there.
\item Algorithmic questions in the ``matrix group recognition project'', 
see Cohen et al. \cite{Coet}, Magaard--Wilson \cite{MaWi} and further
references there.
\end{itemize}
Now, in principle, explicit models for $\fg$ and $G_K$ can be written
down using \cite{Ca1}, \cite{St}, but these are not entirely canonical as 
they involve the choice of certain signs in a Chevalley basis of~$\fg$. 
Very recently, Lusztig \cite{L5} remarked that Chevalley's original
construction of $G_K$ can be simplified using the ``canonical basis''
of the adjoint representation of~$\fg$, in which the action of the
Chevalley generators of $\fg$ is given by explicit and remarkably
simple formulae.

In this note we show that there is an analogous simplification of the 
construction of groups $G_K$ which are not necessarily of adjoint type. 
The starting point is Jantzen's description \cite{J} of a minuscule highest 
weight representation of $\fg$, which we recall in Section~\ref{sec1}. 
This description involves an explicit basis on which the Chevalley
generators of $\fg$ act again by remarkably simple formulae. I~wish to 
thank George Lusztig for pointing out to me that this basis is actually 
the ``canonical basis'' of the module in the sense of \cite{L6}; see 
Proposition~\ref{remhw1}. We will not use this result as such but, of 
course, it provides a background explanation for why the various constructions
work so smoothly. In order to proceed, we then rely on the well-known fact 
that the simple reflections of the Weyl group of $\fg$ are induced by 
certain Lie algebra automorphisms. In Section~\ref{sec2}, we show that 
these automorphisms take a particularly simple form in our setting. 
Finally, in Sections~\ref{sec3} and \ref{sec4}, we discuss the construction
of the corresponding Chevalley groups. 

The construction here actually turns out to be simpler in many ways than that 
of groups of adjoint type. It is ``canonical'' in the sense that it does 
not involve any choices of signs; furthermore, it does not rely at all on 
the use of Kostant's $\Z$-form. The resulting matrix representations of 
$G_K$ are completely explicit and can be easily implemented on a computer. 
As in \cite{my}, we have made a certain attempt to keep the whole argument 
as elementary and self-contained as possible. 

\section{Minuscule weight modules} \label{sec1}
We recall some basic facts about root systems; see, e.g., Carter
\cite[Chap.~2]{Ca1}, Humphreys \cite[Chap.~III]{H}. Let $E$ be a 
finite-dimensional vector space over $\Q$ and $(\;,\;) \colon E\times E
\rightarrow \Q$ be a symmetric bilinear form such that $(e,e)>0$ for all
$0\neq e\in E$. For each $0\neq e\in E$, we denote $e^\vee:=\frac{2}{(e,e)}
e\in E$ and define the corresponding reflection $w_e\colon E\rightarrow E$ 
by $w_e(v)=v-(v,e^\vee)e$ for all $v\in E$. Let $\Phi\subseteq E$ be a 
reduced crystallographic root system. Thus, $\Phi$ is a finite subset of 
$E\setminus \{0\}$ such that $E=\langle \Phi \rangle_\Q$; furthermore, 
the following hold for all $\alpha,\beta\in \Phi$:
\begin{itemize}
\item if $\beta\neq \pm \alpha$, then $\alpha,\beta$ are linearly
independent in $E$;
\item we have $(\beta,\alpha^\vee)\in\Z$ and $w_\alpha(\beta) \in \Phi$.
\end{itemize}
We assume throughout that $\Phi$ is irreducible. Let $\Pi=\{\alpha_i\mid 
i\in I\}$ be a set of simple roots in $\Phi$, where $I$ is a finite index 
set. Then $\Pi$ is a basis of $E$ and every $\alpha\in \Phi$ is a linear 
combination of $\Pi$ where either all coefficients are in $\Z_{\geq 0}$ 
or all coefficients are in $\Z_{\leq 0}$; correspondingly, we have a 
partition $\Phi=\Phi^+\cup \Phi^-$ where $\Phi^+$ are the positive roots 
and  $\Phi^-=-\Phi^+$ are the negative roots. The matrix
\[A=(a_{ij})_{i,j\in I}\qquad \mbox{where} \qquad a_{ij}:=(\alpha_j,
\alpha_i^\vee),\]
is called the Cartan matrix of $\Phi$ with respect to $\Pi$. We have
\[ a_{ii}=2 \qquad \mbox{and} \qquad a_{ij}\in\{0,-1,-2,-3\} \quad 
\mbox{for $i\neq j$}.\]
Furthermore, it is known that $A$ is independent of the choice of $\Pi$, 
up to simultaneous permutation of the rows and columns. Let $W:=\langle 
w_\alpha\mid \alpha\in\Phi\rangle \subseteq \GL(E)$ be the Weyl group of
$\Phi$. Let $S=\{s_i\mid i\in I\}$ where $s_i=w_{\alpha_i}$ for $i\in I$. 
Then $(W,S)$ is a Coxeter system and $\Phi=\{w(\alpha_i)\mid i\in I,
w\in W\}$. We note the following compatibility property.

\begin{rem} \label{rem00} We have $w(e^\vee)=(w(e))^\vee$ for all 
$w\in W$ and $e\in E$. Consequently, we have $(w^{-1}(e'),e^\vee)=
(e',(w(e))^\vee)$ for all $w\in W$ and $e,e'\in E$. 

(This immediately follows from that the fact that $(w(e),w(e))=(e,e)$.)
\end{rem}

We recall some notions concerning ``weights''; see, e.g., Humphreys 
\cite[\S 13]{H}. By definition, the {\em weight lattice} is the $\Z$-lattice 
$\Lambda\subseteq E$ spanned by the basis $\{\varpi_i\mid i\in I\}$ of 
$E$ which is dual to the basis $\{\alpha_i^\vee \mid i\in I\}$. Then $\Phi 
\subseteq \Lambda$ and $w(\Lambda) \subseteq \Lambda$ for all $w\in W$; 
we have $s_i(\varpi_j)=\varpi_j-\delta_{ij}\alpha_i$ for all $i,j\in I$.

A weight $\lambda\in\Lambda$ is called {\em dominant} if $(\lambda,
\alpha_i^\vee) \geq 0$ for all $i\in I$. Each $W$-orbit on $\Lambda$ 
contains exactly one dominant weight. Now let $0\neq \lambda\in\Lambda$ be 
dominant. Then $\lambda$ is called {\em minuscule} if $(\lambda,
\alpha^\vee) \in \{0,\pm 1\}$ for all $\alpha\in \Phi$. We then also have 
$(w(\lambda), \alpha^\vee)\in \{0, \pm 1\}$ for all $w\in W$ and 
$\alpha\in\Phi$. Furthermore, $\lambda= \varpi_i$ for some $i\in I$. The 
possibilities for the various types of $\Phi$ are listed in 
Table~\ref{Mcoxgraphs}; note that there are no minuscule dominant weights 
for $W$ of type $G_2$, $F_4$, $E_8$. (For all this, see \cite[VIII, \S 7, 
no.~3]{B} and \cite[Exc.~13.13]{H}; Humphreys uses the terminology 
``non-zero minimal dominant weight'' instead of ``minuscule dominant 
weight''.) 

\begin{table}[htbp]
\caption{Minuscule weights (marked by ``$\circ$'')}
\label{Mcoxgraphs}
\begin{center}
\begin{picture}(345,125)
\put( 15, 25){$E_6$}
\put( 40, 25){\circle{6}}
\put( 38, 31){$1$}
\put( 43, 25){\line(1,0){20}}
\put( 60, 25){\circle*{5}}
\put( 58, 31){$3$}
\put( 60, 25){\line(1,0){20}}
\put( 80, 25){\circle*{5}}
\put( 78, 31){$4$}
\put( 80, 25){\line(0,-1){20}}
\put( 80,  5){\circle*{5}}
\put( 85,  3){$2$}
\put( 80, 25){\line(1,0){20}}
\put(100, 25){\circle*{5}}
\put( 98, 31){$5$}
\put(100, 25){\line(1,0){17}}
\put(120, 25){\circle{6}}
\put(118, 31){$6$}

\put(215, 25){$E_7$}
\put(240, 25){\circle*{5}}
\put(238, 31){$1$}
\put(240, 25){\line(1,0){20}}
\put(260, 25){\circle*{5}}
\put(258, 31){$3$}
\put(260, 25){\line(1,0){20}}
\put(280, 25){\circle*{5}}
\put(278, 31){$4$}
\put(280, 25){\line(0,-1){20}}
\put(280,  5){\circle*{5}}
\put(285,  3){$2$}
\put(280, 25){\line(1,0){20}}
\put(300, 25){\circle*{5}}
\put(298, 31){$5$}
\put(300, 25){\line(1,0){20}}
\put(320, 25){\circle*{5}}
\put(318, 31){$6$}
\put(320, 25){\line(1,0){17}}
\put(340, 25){\circle{6}}
\put(338, 31){$7$}

\put( 13,60){$C_n$}
\put( 12,49){$\scriptstyle{n \geq 2}$}
\put( 40,55){\circle*{5}}
\put( 38,61){$1$}
\put( 40,56){\line(1,0){20}}
\put( 40,53){\line(1,0){20}}
\put( 46,52){$>$}
\put( 60,55){\circle*{5}}
\put( 58,61){$2$}
\put( 60,55){\line(1,0){20}}
\put( 80,55){\circle*{5}}
\put( 78,61){$3$}
\put( 80,55){\line(1,0){10}}
\put(100,55){\circle*{1}}
\put(110,55){\circle*{1}}
\put(120,55){\circle*{1}}
\put(130,55){\line(1,0){27}}
\put(140,55){\circle*{5}}
\put(128,61){$n{-}1$}
\put(148,55){\line(1,0){9}}
\put(160,55){\circle{6}}
\put(158,61){$n$}

\put( 13,90){$B_n$}
\put( 12,79){$\scriptstyle{n \geq 2}$}
\put( 40,85){\circle{6}}
\put( 38,91){$1$}
\put( 43,83){\line(1,0){16}}
\put( 43,86){\line(1,0){16}}
\put( 46,82){$<$}
\put( 60,85){\circle*{5}}
\put( 58,91){$2$}
\put( 60,85){\line(1,0){30}}
\put( 80,85){\circle*{5}}
\put( 78,91){$3$}
\put(100,85){\circle*{1}}
\put(110,85){\circle*{1}}
\put(120,85){\circle*{1}}
\put(130,85){\line(1,0){10}}
\put(140,85){\circle*{5}}
\put(128,91){$n{-}1$}
\put(140,85){\line(1,0){20}}
\put(160,85){\circle*{5}}
\put(158,91){$n$}

\put( 13,120){$A_{n}$}
\put( 12,109){$\scriptstyle{n \geq 1}$}
\put( 40,115){\circle{6}}
\put( 38,121){$1$}
\put( 43,115){\line(1,0){14}}
\put( 60,115){\circle{6}}
\put( 58,121){$2$}
\put( 63,115){\line(1,0){14}}
\put( 80,115){\circle{6}}
\put( 77,121){$3$}
\put( 83,115){\line(1,0){11}}
\put(100,115){\circle*{1}}
\put(110,115){\circle*{1}}
\put(120,115){\circle*{1}}
\put(128,115){\line(1,0){9}}
\put(140,115){\circle{6}}
\put(128,121){$n{-}1$}
\put(143,115){\line(1,0){14}}
\put(160,115){\circle{6}}
\put(158,121){$n$}

\put(200,87){$D_n$}
\put(200,76){$\scriptstyle{n \geq 3}$}
\put(220,105){\circle{6}}
\put(225,105){$2$}
\put(220,65){\circle{6}}
\put(226,60){$1$}
\put(222,103){\line(1,-1){18}}
\put(222,67){\line(1,1){18}}
\put(240,85){\circle*{5}}
\put(238,91){$3$}
\put(240,85){\line(1,0){30}}
\put(260,85){\circle*{5}}
\put(258,91){$4$}
\put(280,85){\circle*{1}}
\put(290,85){\circle*{1}}
\put(300,85){\circle*{1}}
\put(308,85){\line(1,0){10}}
\put(320,85){\circle*{5}}
\put(308,91){$n{-}1$}
\put(320,85){\line(1,0){17}}
\put(340,85){\circle{6}}
\put(338,91){$n$}
\end{picture}
\end{center}
\end{table}

\begin{defn}[Jantzen \protect{\cite[5A.1]{J}}] \label{defj}
Assume that $\Psi\subseteq \Lambda$ is a non-empty union of $W$-orbits
of minuscule weights. (In particular, $\Psi$ is finite and the Dynkin 
diagram of $\Phi$ is one of the graphs in Table~\ref{Mcoxgraphs}.) 
Let $M$ be a $\C$-vector space with basis $\{z_\mu \mid \mu \in \Psi\}$. 
We define linear maps
\[e_i\colon M\rightarrow M, \qquad f_i\colon M\rightarrow M,\qquad
h_i\colon M\rightarrow M\qquad (i\in I)\]
by the following formulae, where $\mu\in \Psi$:
\begin{align*}
e_i(z_\mu)&:=\left\{\begin{array}{cl} z_{\mu+\alpha_i} & \quad 
\mbox{if $(\mu,\alpha_i^\vee)=-1$},\\0 & \quad \mbox{otherwise},
\end{array}\right.\\
f_i(z_\mu)&:=\left\{\begin{array}{cl} z_{\mu-\alpha_i} & \quad 
\mbox{if $(\mu,\alpha_i^\vee)=1$},\\0 & \quad \mbox{otherwise},
\end{array}\right.\\
h_i(z_\mu)&:=(\mu,\alpha_i^\vee)z_\mu.
\end{align*}
These maps are well-defined: if $(\mu,\alpha_i^\vee)=-1$, then $\mu+\alpha_i=
s_i(\mu)\in \Psi$; similarly, if $(\mu,\alpha_i^\vee) =1$, then $\mu-
\alpha_i=s_i(\mu)\in\Psi$. (Note that Jantzen actually deals with the 
quantum group case, which gives rise to a number of technical complications 
which are not present in our setting here.)
\end{defn}

\begin{rem} \label{rem10} As in \cite[\S 10.1]{H}, we have a partial 
order on $E$ defined by $e\preceq e'$ if $e=e'$ or $e'-e$ is a sum of 
positive roots. Let us choose an enumeration $\Psi=\{\mu_1,\ldots,
\mu_d\}$ such that  $i\leq j$ whenever $\mu_j\preceq \mu_i$. Then
the above formulae show that each $e_i$ is represented by a strictly 
upper triangular matrix and each $f_i$ is represented by a strictly 
lower triangular matrix. In particular, $e_i,f_i$ are nilpotent linear 
maps; in fact, one immediately checks that $e_i^2=f_i^2=0$ for 
$i\in I$. Also note that, clearly, each $h_i$ is represented by a 
diagonal matrx.
\end{rem}

\begin{lem} \label{rem00a} Let $\Psi\subseteq \Lambda$ be as in
Definition~\ref{defj}. Then $\Phi\subseteq \Z\Psi$. More precisely,
for each $i\in I$, there exists some $\mu\in\Psi$ such that 
$(\mu,\alpha_i^\vee)\neq 0$. (In particular, we have $h_i\neq 0$.) Then 
$s_i(\mu)= \mu\pm \alpha_i \in \Psi$ and so $\pm \alpha_i=s_i(\mu)-
\mu\in\Z\Psi$.  
\end{lem}

\begin{proof} Let $i_0\in I$ be such that $\varpi_{i_0}\in\Psi$ is 
minuscule. If $i=i_0$, then $(\varpi_{i_0},\alpha_i^\vee)=1$ and so 
the assertion is clear. Now assume that $i\neq i_0$. We claim that there
exists some $w\in W$ such that $(w(\varpi_{i_0}),\alpha_i^\vee)\neq 0$.
Indeed, we can find a path in the Dynkin diagram of $\Phi$ connecting 
$i_0$ and~$i$. Let $i_1,\ldots,i_k\in I$ ($k\geq 0$) be such that $i_0,
i_1,\ldots, i_k,i$ label the nodes along this path. Then $s_{i_0}
(\varpi_{i_0})=\varpi_{i_0}-\alpha_{i_0}$; furthermore, $s_{i_1}
(\varpi_{i_0})=\varpi_{i_0}$ and so $s_{i_1}s_{i_0}(\varpi_{i_0})=
\varpi_{i_0}-s_{i_1}(\alpha_{i_0})=\varpi_{i_0}-\alpha_{i_0}-
c_1\alpha_{i_1}$ where $c_1\neq 0$ since $i_0,i_1$ are joined by an edge 
in the graph. Next, since $i_0,i_2$ are not joined by an edge, we have
$s_{i_2}(\alpha_{i_1})=\alpha_{i_1}$ and so $s_{i_2}s_{i_1}s_{i_0}
(\varpi_{i_0})=-\varpi_{i_0}-\alpha_{i_0}-c_1\alpha_{i_1}-c_2\alpha_{i_2}$ 
where $c_2\neq 0$ since $i_1,i_2$ are joined by an edge. Continuing in this 
way we find that $s_{i_k}\cdots s_{i_1}s_{i_0} (\varpi_{i_0})=-\varpi_{i_0}
-\alpha_{i_0}-c_1\alpha_{i_1}-\ldots -c_k\alpha_{i_k}$ with $c_j\neq 0$
for $1\leq j\leq k$. But then $(s_{i_k}\cdots s_{i_1}s_{i_0}(\varpi_{i_0}),
\alpha_i^\vee)\neq 0$ since $\alpha_i$ is joined to $\alpha_{i_k}$ by
an edge but not joined to any $\alpha_{i_0},\ldots, \alpha_{i_{k-1}}$.
(Note that the Dynkin diagram of $\Phi$ is a tree.)
\end{proof}

We now consider the Lie algebra $\gl(M)$, with the usual Lie bracket. 

\begin{lem}[Cf.\ \protect{\cite[5A.1]{J}}] \label{lemj} The elements
$\{h_i\mid i\in I\}$ are linearly independent in $\gl(M)$ and 
commute with each other. Furthermore, we have the relations:
\begin{itemize}
\item[(a)] $[h_j,e_i]=(\alpha_j,\alpha_i^\vee)e_i$ and
$[h_j,f_i]=-(\alpha_j,\alpha_i^\vee)f_i$ for all $i,j\in I$.
\item[(b)] $[e_i,f_i]=h_i$ for all $i\in I$ and $[e_i,f_j]=0$ for 
all $i\neq j$ in $I$.
\item[(c)] We have $\operatorname{ad}(e_i)^{1-a_{ij}}(e_j)=
\operatorname{ad}(f_i)^{1-a_{ij}}(f_j)=0$ for all $i\neq j$ in $I$.
\end{itemize}
\end{lem}

\begin{proof} Since each $h_i$ is represented by a diagonal matrix,
it is clear that $[h_i,h_j]=0$ for all $i,j$. Assume that we have a relation
$\sum_{j\in I} x_jh_j=0$ where $x_j\in \C$. Then we obtain
\[0=\sum_{j\in I} x_jh_j(z_\mu) =\Bigl(\sum_{j\in I} x_j(\mu,
\alpha_j^\vee)\Bigr)z_\mu\qquad \mbox{for all $\mu\in\Psi$}.\]
By Lemma~\ref{rem00a}, $\Phi\subseteq \Z\Psi$ and so $E=\langle\Psi
\rangle_\R$. This forces $x_j=0$ for all $j\in I$.

(a) Let $\mu\in \Psi$. If $(\mu,\alpha_i^\vee)=-1$, then $e_i(z_\mu)=
z_{\mu+\alpha_i}$ and so $[h_j,e_i](z_\mu)=h_j(z_{\mu+\alpha_i})-
(\mu,\alpha_j^\vee)z_{\mu+\alpha_i}=(\alpha_i,\alpha_j^\vee)e_i(z_\mu)$, as 
required. If $(\mu,\alpha_i^\vee)\neq -1$, then both $[h_j,e_i](z_\mu)$
and $(\alpha_i,\alpha_j^\vee)e_i(z_\mu)$ are equal to~$0$. Thus,
$[h_j,e_i]=(\alpha_j,\alpha_i^\vee)e_i$. The argument for proving $[h_j,
f_i]=-(\alpha_j,\alpha_i^\vee)f_i$ is analogous. 

(b) Let $i\in I$ and $\mu\in \Psi$. If $(\mu,\alpha_i^\vee)=-1$, then 
$f_{i}(z_\mu)=0$ and so $[e_i,f_i](z_\mu)=-f_i e_i(z_\mu)=-f_i
(z_{\mu+\alpha_i})=-z_\mu$, since $(\mu+\alpha_i,\alpha_i^\vee)=1$. On 
the other hand, $h_i(z_\mu)=(\mu,\alpha_i^\vee)z_\mu=-z_\mu$, as required. 
If $(\mu,\alpha_i^\vee)=1$, the argument is analogous. If $(\mu,
\alpha_i^\vee)=0$, then both $[e_i,f_i](z_\mu)$ and $(\mu,\alpha_i^\vee)
h_i(z_\mu)$ are~$0$. Thus, $[e_i,f_i]=h_i$. Now let $j\in I$, $i\neq j$. 
We must show that $e_if_j(z_\mu)=f_je_i(z_\mu)$. We have 
\begin{align*}
f_je_i(z_\mu)&=\left\{\begin{array}{cl}
z_{\mu+\alpha_i-\alpha_j} & \; \mbox{if $(\mu,\alpha_i^\vee)=-1$ and 
$(\mu+\alpha_i,\alpha_j^\vee) =1$},\\ 0 & \;\mbox{otherwise};
\end{array}\right.\\
e_if_j(z_\mu)&=\left\{\begin{array}{cl}
z_{\mu+\alpha_i-\alpha_j} & \; \mbox{if $(\mu,\alpha_j^\vee)=1$ and 
$(\mu-\alpha_j,\alpha_i^\vee)=-1$},\\ 0 & \;\mbox{otherwise}.
\end{array}\right.
\end{align*}
So we only need to show that the conditions on the right hand side are
equivalent. Assume first that $(\mu,\alpha_i^\vee)=-1$ and 
$(\mu+\alpha_i,\alpha_j^\vee) =1$. Since $i\neq j$, we have
$(\alpha_i,\alpha_j^\vee)\leq 0$ and so $(\mu,\alpha_j^\vee)=
1-(\alpha_i,\alpha_j^\vee)\geq 1$. Hence, we must have $(\mu,\alpha_j^\vee)
=1$ and $(\alpha_i,\alpha_j^\vee)=0$. This then also yields that 
$(\mu-\alpha_j,\alpha_i^\vee)=-1$, as required. The reverse implication
is proved similarly.

(c) First we show that $\mbox{ad}(e_i)^{1-a_{ij}}(e_j)=0$ for
$i\neq j$. Let $\mu\in \Psi$. We have
\begin{align*}
e_ie_j(z_\mu)&=\left\{\begin{array}{cl} z_{\mu+\alpha_i+\alpha_j}
&  \quad \mbox{if $(\mu,\alpha_j^\vee)=-1$ and $(\mu+\alpha_j,\alpha_i^\vee)
=-1$},\\ 0 & \quad \mbox{otherwise};\end{array}\right.\\
e_je_i(z_\mu)&=\left\{\begin{array}{cl} z_{\mu+\alpha_i+\alpha_j}
&  \quad \mbox{if $(\mu,\alpha_i^\vee)=-1$ and $(\mu+\alpha_i,\alpha_j^\vee)
=-1$},\\ 0 & \quad \mbox{otherwise}.\end{array}\right.
\end{align*}
Now $(\mu+\alpha_j,\alpha_i^\vee)=(\mu,\alpha_i^\vee)+a_{ij}$ and 
$(\mu+\alpha_i,\alpha_j^\vee)=(\mu,\alpha_j^\vee)+a_{ji}$. If $a_{ij}=
a_{ji}=0$, then $e_ie_j(z_\mu)=e_je_i(z_\mu)$ for all $\mu\in\Psi$ and 
so $\mbox{ad}(e_i)(e_j)=[e_i,e_j]=0$, as required. Now assume that
$a_{ij}<0$. We have $\mbox{ad}(e_i)^2(e_j)=[e_i,[e_i,e_j]]=e_i^2e_j-2e_i
e_je_i+e_je_i^2=-2e_ie_je_i$ since $e_i^2=0$; see Remark~\ref{rem10}. 
If $e_ie_je_i=0$, then $\mbox{ad}(e_i)^2(e_j)=0$, as required. Finally,
assume that $e_ie_je_i(z_\mu)\neq 0$ for some $\mu\in\Psi$. Then $(\mu,
\alpha_i^\vee)=(\mu+\alpha_i, \alpha_j^\vee)=(\mu+\alpha_i+\alpha_j, 
\alpha_i^\vee)=-1$, which implies that $a_{ij}=-2$. But then 
$\mbox{ad}(e_i)^{1-a_{ij}}(e_j)=\mbox{ad}(e_i)^3(e_j)=-2[e_i,e_ie_j
e_i]=0$. 

The argument for the relation $\mbox{ad}(f_i)^{1-a_{ij}}(f_j)=0$ is 
completely analogous.
\end{proof}

Let $\fg \subseteq \gl(M)$ be the Lie subalgebra generated by the maps 
$e_i,f_i$ ($i\in I$). Let $\fh:=\langle h_i \mid i\in I\rangle_\C$; by 
Lemma~\ref{lemj}, this is an abelian Lie subalgebra of $\fg$. Let $\fh^*=
\Hom(\fh,\C)$. For $\mu\in \Lambda$, we define $\dmu \in\fh^*$ 
by $\dmu(h_j):=(\mu,\alpha_j^\vee)$ for all $j\in I$. For any $\lambda 
\in\fh^*$, we define as usual $\fg_\lambda:=\{x\in\fg \mid [h,x]=
\lambda(h)x \mbox{ for all $h\in \fh$}\}$. If $\lambda=\dmu$, we also 
write $\fg_\mu$ instead of $\fg_{\dmu}$; thus, $e_i\in\fg_{\alpha_i}$ 
and $f_i\in\fg_{-\alpha_i}$ for $i\in I$.

\begin{prop} \label{propj} Recall that we fixed a set $\Psi\subseteq 
\Lambda$ as in Definition~\ref{defj}. Then, with the above notation, 
$\fh\subseteq \fg$ is a Cartan subalgebra and $\fg$ is simple with 
corresponding root system $\dPhi:=\{\dalp\mid \alpha \in\Phi\}$. In 
particular, we have a direct sum decomposition $\fg =\fh \oplus 
\bigoplus_{\alpha \in\Phi} \fg_{\alpha}$ where $\dim \fg_{\alpha}=1$ for 
$\alpha\in\Phi$.
\end{prop}

\begin{proof} Since the relations in Lemma~\ref{lemj} hold, Serre's Theorem 
\cite[\S 18.3]{H} shows that $e_i,f_i,h_i$ define a representation of a 
semisimple Lie algebra $\tilde{\fg}$ with root system isomorphic to 
$\Phi$. Since $\Phi$ is irreducible, $\tilde{\fg}$ is simple and so 
that representation must be injective. Hence, we obtain an isomorphism
$\tilde{\fg}\cong \fg$ under which the Cartan subalgebra of $\tilde{\fg}$
is mapped onto $\fh$.~---~Alternatively, one could also argue as in 
\cite[\S 4]{my}: We have $\Psi=\Psi_1\cup \ldots \cup \Psi_r$ where
the $\Psi_i$ are the $W$-orbits on $\Psi$. Then $M=M_1\oplus \ldots 
\oplus M_r$ where $M=\langle z_\mu\mid \mu\in\Psi_i\rangle_\C$, hence
$\fg\subseteq \mathfrak{sl}(M_1)\oplus \ldots \oplus \mathfrak{sl}(M_r)$.
Then one shows that each $M_i$ is a simple $\fg$-module and uses a general 
semisimplicity criterion for $\fg$; see \cite[\S 19.1]{H}. 

The statements about the direct sum decomposition of $\fg$ are classical
facts about semisimple Lie algebras; see, e.g., \cite[\S 8.4]{H}.
\end{proof}

\begin{rem} \label{tract} The map $\mu\mapsto \dmu$ ($\mu\in\Lambda$)
defines a linear isomorphism $E\stackrel{\sim}{\rightarrow} \fh^*$. Via 
this isomorphism, we now identify $E$ with $\fh^*$. Thus, $W \subseteq 
\GL(\fh^*)$ where 
\[s_i(\dmu)=\dmu-(\mu,\alpha_i^\vee) \dalp_i\qquad \mbox{for all 
$i\in I$ and $\mu\in \Lambda$}.\]
For $w\in W$, let $w^*\colon \fh \rightarrow \fh$ be the transposed map, 
that is, we have $\lambda \circ w^*=w\circ \lambda$ for all $\lambda \in 
\fh^*$. Let $i\in I$. Then a straightforward computation shows that 
\[\dmu(s_i^*(h_j))=(s_i(\dmu))(h_j)=\dmu(h_j-(\alpha_i,\alpha_j^\vee)h_i) 
\qquad \mbox{for $\mu\in \Lambda$, $j\in I$}.\]
Thus, the map $s_i^*\colon \fh\rightarrow \fh$ is given by $h_j\mapsto h_j-
(\alpha_i,\alpha_j^\vee)h_i$ ($j \in I$). 
\end{rem}

\begin{rem} \label{remhw} Assume that $\Psi$ is a single $W$-orbit
of a minuscule weight $\varpi_j$, where $j\in I$ is
one of the nodes marked ``$\circ$'' in Table~\ref{Mcoxgraphs}.
Then we have $e_i(z_{\varpi_j})=0$ and $h_i(z_{\varpi_j})=(\varpi_j,
\alpha_i^\vee)z_{\varpi_j}=\delta_{ij}z_{\varpi_j}$ for all $i\in I$. 
Thus, $z_{\varpi_j}\in M$ is a primitive vector, with corresponding 
weight~$\dot{\varpi}_j\in\fh^*$. Hence, $M$ is a highest weight module 
with highest weight $\dot{\varpi}_j$.  (See \cite[Chap.~VI]{H} for
these notions.) 
\end{rem}

\begin{prop}[G. Lusztig] \label{remhw1} Let $\Psi$ be as in 
Remark~\ref{remhw}. Then $\{z_\mu \mid \mu \in \Psi\}$ is the 
canonical basis of $M$ in the sense of \cite[\S 14.4]{L6}.
\end{prop}

\begin{proof} Let $\mu\in\Psi$ and write $\mu=w(\varpi_j)$ where $w\in W$.
We choose a reduced expression
$w=s_{i_1}\cdots s_{i_k}$ where $k\geq 0$ and $i_1,\ldots,i_k\in I$. Let 
\[\theta_w(\varpi_j):=f_{i_1}^{a_1}\cdots f_{i_k}^{a_k}(z_{\varpi_j})\in M\]
where $a_1:=(s_{i_2}\cdots s_{i_k}(\varpi_j),\alpha_{i_1}^\vee)$, 
$a_2:=(s_{i_3}\cdots s_{i_k}(\varpi_j),\alpha_{i_2}^\vee)$, $\ldots$, 
$a_k:=(\varpi_j,\alpha_{i_k}^\vee)$, as in \cite[\S 28.1]{L6}; note that
$a_1,\ldots,a_k\in \Z_{\geq 0}$ and $\theta_w(\varpi_j)$ only depends 
on $w,\varpi_j$. Now a simple induction on $k$ shows that $\theta_w
(\varpi_j)=z_{w(\varpi_j)}=z_\mu$. Hence, the fact that $z_\mu$ belongs to 
the canonical basis of $M$ is a very special case of \cite[Prop.~28.1.4]{L6}.
\end{proof}

\section{Weyl group action} \label{sec2}

We keep the setting of the previous section. Our first aim is to 
``lift'' the induced action of a generator $s_i\in W$ on the Cartan 
subalgebra $\fh \subseteq \fg$ (see Remark~\ref{tract}) to a suitable
automorphism of $\fg$. This is well-known in the general theory of 
semisimple Lie algebras; see, e.g., \cite[Chap.~VIII, \S 2, Lemme~1]{B}. 
We show that in our setting, these automorphisms take a 
particularly simple form.

Let $i\in I$. By Remark~\ref{rem10}, we have $e_i^2=0$ and so $\mbox{id}_M
+ te_i\in \GL(M)$ for any $t\in\C$. Similarly, $f_i^2=0$ and so 
$\mbox{id}_M+tf_i\in \GL(M)$ for any $t\in\C$. We set 
\[n_i(t):=(\mbox{id}_M+te_i)(\mbox{id}_M-t^{-1}f_i)(\mbox{id}_M+te_i) 
\in\GL(M) \quad \mbox{where $0\neq t\in\C$}.\]
Let $x\in\fg$. Then an easy computation shows that $[e_i,[e_i,x]]=
-2e_ixe_i$ and 
\begin{center}
$(\mbox{id}_M+te_i)x(\mbox{id}_M+te_i)^{-1}=x+t[e_i,x]+\frac{1}{2}t^2
[e_i,[e_i,x]]\in \fg$.
\end{center}
Now note that conjugation with any element of $\GL(M)$ defines a Lie 
algebra automorphism of $\gl(M)$. The above formula shows that conjugation
with $\mbox{id}_M+te_i\in\GL(M)$ restricts to a Lie algebra automorphism 
of $\fg$. A similar statement holds for conjugation with $\mbox{id}_M+tf_i$
and, hence, also for $n_i(t)$; thus, $n_i(t)\fg n_i(t)^{-1}\subseteq \fg$. 

\begin{lem} \label{lemni} Let $i\in I$ and $0\neq t\in\C$. Then
$n_i(t)^{-1}=n_i(-t)$ and 
\[ n_i(t)(z_\mu)=\left\{\begin{array}{cl} z_\mu & \quad 
\mbox{if $(\mu, \alpha_i^\vee)=0$},\\ -t^{-1}z_{\mu-\alpha_i} & 
\quad \mbox{if $(\mu,\alpha_i^\vee)=1$},\\ tz_{\mu+\alpha_i} & 
\quad \mbox{if $(\mu, \alpha_i^\vee)=-1$}. \end{array}\right.\]
Setting $n_i:=n_i(1)$, we have $n_i(z_\mu)=\pm z_{s_i(\mu)}$ 
and $n_i^2(z_\mu)=(-1)^{(\mu,\alpha_i^\vee)} z_\mu$. Note that
$n_i$ is represented by a monomial matrix with non-zero entries 
equal to $\pm 1$.
\end{lem}

\begin{proof} The formula for $n_i(t)^{-1}$ is clear; just recall that
$e_i^2=f_i^2=0$. Now let $\mu\in\Psi$. A straightforward computation 
yields that
\[ n_i(t)(z_\mu)=z_\mu+2te_i(z_\mu)-t^{-1}f_i(z_\mu)-e_if_i(z_\mu)- 
f_ie_i(z_\mu)-te_if_ie_i (z_\mu).\]
If $(\mu,\alpha_i^\vee)=0$, then this immediately shows that $n_i(t)
(z_\mu)=z_\mu$. Now assume that $(\mu,\alpha_i^\vee)=1$. Then
$e_i(z_\mu)=0$ and $f_i(z_\mu)=z_{\mu-\alpha_i}$, hence
the above expression simplifies to $n_i(t)(z_\mu)=z_\mu-t^{-1}
z_{\mu-\alpha_i}-e_i(z_{\mu-\alpha_i})$. We have $(\mu-\alpha_i,
\alpha_i^\vee)=-1$ and so $e_i(z_{\mu-\alpha_i})=z_\mu$. This
yields $n_i(t)(z_\mu)=-t^{-1}z_{\mu-\alpha_i}$, as claimed.
Finally, assume that $(\mu,\alpha_i^\vee)=-1$. Then $f_i(z_\mu)=0$
and $e_i(z_\mu)=z_{\mu+\alpha_i}$, hence the above expression
simplifies to $n_i(t)(z_\mu)=z_\mu+2tz_{\mu+\alpha_i}-
f_i(z_{\mu+\alpha_i})- te_if_i(z_{\mu+\alpha_i})$. Now
$(\mu+\alpha_i,\alpha_i^\vee)=1$ and so $f_i(z_{\mu+\alpha_i})=
z_\mu$. This yields $n_i(t)(z_\mu)= tz_{\mu+\alpha_i}$,
as claimed. The formula for $n_i^2$ is an immediate consequence.
\end{proof}

The following result about braid relations can be found in 
\cite[Lemma~56 (p.~149)]{St}, as a {\it consequence} of the main 
structural properties of Chevalley groups (e.g., the $BN$-pair axioms). 
Here, we can prove it directly based on the explicit formulae in 
Lemma~\ref{lemni} (see also \cite[2.4]{L3} and \cite[Prop.~9.3.2]{Sp}). 
In our setting, the braid relations will then be a useful tool in the 
discussion in Section~\ref{sec3}. 

\begin{prop} \label{braidr} The elements $n_i$ ($i\in I$)
satisfy the {\it braid relations}: let $i,j\in I$, $i\neq j$, and 
$m\geq 2$ be the order of $s_is_j\in W$. Then 
\[n_in_jn_i\cdots=n_jn_in_j \cdots\qquad \mbox{(with $m$ factors on 
both sides)}.\]
\end{prop}

\begin{proof} Since the Dynkin diagram of $\Phi$ is as in 
Table~\ref{Mcoxgraphs}, we have $m\in\{2,3,4\}$. Let us first assume
that $m=2$. Then $a_{ij}=(\alpha_j,\alpha_i^\vee)=0$ and so $[e_i,e_j]=
[f_i,f_j]=0$ by Lemma~\ref{lemj}(c). Since we also have $[e_i,f_j]=
[f_i,e_j]=0$ by Lemma~\ref{lemj}(b), the defining formula for $n_i,
n_j$ immediately shows that $n_in_j=n_jn_i$, as required. 

Now assume that $m\in\{3,4\}$. If $m=3$, then $a_{ij}=a_{ji}=-1$; if $m=4$, 
then $\{a_{ij},a_{ji}\}=\{-1,-2\}$ and we choose the notation such that
$a_{ji}=-1$. Hence, in both cases, $a_{ji}=-1$ and so 
\begin{equation*}
s_i(\alpha_j^\vee)=\alpha_j^\vee-(\alpha_j^\vee,\alpha_i^\vee)\alpha_i=
\alpha_j^\vee-a_{ji}\alpha_i^\vee=\alpha_i^\vee+\alpha_j^\vee.\tag{$*$}
\end{equation*}
Now let $W'=\langle s_i,s_j \rangle \subseteq W$; then $W'$ is a dihedral 
group of order $6$ or~$8$. Since $\Psi$ is a union of $W$-orbits on 
$\Lambda$, we can decompose $\Psi$ as a union of $W'$-orbits. 
Correspondingly, we have $M=\bigoplus_\cO M_\cO$ where $\cO$ runs over 
the $W'$-orbits on $\Psi$ and $M_\cO:=\langle z_\mu\mid \mu\in\cO
\rangle_\C$. By the formulae in Lemma~\ref{lemni}, it is clear that
$n_i(M_\cO)\subseteq M_\cO$ and $n_j(M_\cO)\subseteq M_\cO$, so it is 
enough to prove the desired identity upon restriction to~$M_\cO$, for 
any~$\cO$.  Now let us fix such a $W'$-orbit $\cO\subseteq \Psi$. 
If $\mu\in\cO$ is such that $(\mu,\alpha_i^\vee)=(\mu,\alpha_j^\vee)=0$,
then $\cO=\{\mu\}$, $n_i(z_\mu)=n_j(z_\mu)=z_\mu$ and so the desired
identity is clear on $M_\cO$. So we can now assume that $(\mu,
\alpha_i^\vee)\neq 0$ or $(\mu,\alpha_j^\vee)\neq 0$, for all $\mu \in \cO$.
We claim that there is some $\mu\in \cO$ such that 
\begin{equation*}
\varepsilon:=(\mu,\alpha_i^\vee)=\pm 1 \quad \mbox{and}\quad 
(\mu,\alpha_j^\vee)=0 \quad \mbox{(and still $a_{ji}=(\alpha_i,
\alpha_j^\vee)=-1$)}. \tag{$*^\prime$}
\end{equation*}
This is seen as follows. Assume that $\nu \in \cO$ is such that $(\nu,
\alpha_j^\vee)\neq 0$. If we also have $(\nu,\alpha_i^\vee)\neq 0$, then 
$(s_i(\nu), \alpha_j^\vee)=(\nu,s_i(\alpha_j^\vee))=(\nu,\alpha_j^\vee)+(\nu,
\alpha_i^\vee)$, using ($*$). Since the left hand side is in $\{0,\pm 1\}$, 
the two terms $(\nu,\alpha_i^\vee)$ and $(\nu, \alpha_j^\vee)$ can not 
be equal. Since they both are $\pm 1$, we conclude that $(\mu,
\alpha_j^\vee)=0$ for $\mu:=s_i(\nu)\in \cO$. On the other hand,
assume that $(\nu,\alpha_i^\vee)=0$. If $m=3$, then we can simply exchange 
the roles of $n_i$ and $n_j$; if $m=4$, then $(s_j(\nu),\alpha_i^\vee)=
(\nu,\alpha_i^\vee)+2(\nu,\alpha_j^\vee)=2(\nu,\alpha_j^\vee)=\pm 2$, a 
contradiction. Thus, ($*^\prime$) is proved. 

Now, if $m=3$, then $\cO=\{\mu,s_i(\mu),s_js_i(\mu)\}$. Using ($*^\prime$)
and the formulae in Lemma~\ref{lemni}, it is straightforward to determine 
the action of $n_i,n_j$ on $M_\cO$; the matrices with respect to the basis
$\{z_\mu,z_{s_i(\mu)},z_{s_js_i(\mu)}\}$ of $M_\cO$ are given by 
\[n_i\colon \left(\begin{array}{ccc} 0 & \varepsilon & 0 \\ -\varepsilon 
& 0 & 0 \\ 0 & 0 & 1 \end{array}\right), \qquad n_j\colon \left(
\begin{array}{ccc} 1 & 0 & 0 \\ 0 & 0 & \varepsilon\\ 0 & -\varepsilon & 0
\end{array}\right).\]
(Note that, in addition to ($*$), we have $s_j(\alpha_i^\vee)=
\alpha_i^\vee-a_{ij}\alpha_j^\vee=\alpha_i^\vee+\alpha_j^\vee$ in this case.)
Then it is a matter of a simple matrix multiplication to check that $n_in_j
n_i= n_jn_in_j$ on $M_\cO$. Similarly, if $m=4$, then $\cO=\{\mu,s_i(\mu),
s_js_i(\mu),s_is_j s_i(\mu)\}$ and we find the following matrices for the
action of $n_i,n_j$ on $M_\cO$:
\[n_i\colon \left(\begin{array}{cccc} 0 & \varepsilon & 0 & 0\\ -\varepsilon 
& 0 & 0 & 0 \\ 0 & 0 & 0 & \varepsilon \\0 & 0 & -\varepsilon & 0
\end{array}\right), \qquad n_j\colon \left(\begin{array}{cccc} 1 & 0 & 0 & 0 
\\ 0 & 0 & \varepsilon & 0 \\ 0 & -\varepsilon & 0 & 0 \\ 0 & 0 & 0 & 1
\end{array}\right).\]
(Note that, in addition to ($*$), $s_j(\alpha_i^\vee)=\alpha_i^\vee-
a_{ij}\alpha_j^\vee=\alpha_i^\vee+2\alpha_j^\vee$ in this case.)
Again, by a simple verification, one checks that $n_in_jn_in_j=
n_jn_in_jn_i$ on $M_\cO$. 
\end{proof}

%
%

\begin{lem} \label{lemni1} 
{\rm (a)} We have $n_i^{-1}h_jn_i=s_i^*(h_j)$ 
for all $i,j\in I$. 

{\rm (b)} We have $n_i\fg_{\alpha}n_i^{-1}=\fg_{s_i(\alpha)}$ for all 
$i\in I$ and $\alpha\in \Phi$. 
\end{lem}


\begin{proof} (a) By Lemma~\ref{lemni}, we have for $\mu\in \Psi$:
\[ h_jn_i(z_\mu)=\left\{\begin{array}{cl} (\mu,\alpha_j^\vee)z_\mu
& \quad \mbox{if $(\mu,\alpha_i^\vee)=0$},\\ -(\mu-\alpha_i,\alpha_j^\vee)
z_{\mu-\alpha_i}  & \quad \mbox{if $(\mu,\alpha_i^\vee)=1$},\\ 
(\mu+\alpha_i,\alpha_j^\vee)z_{\mu+\alpha_i} & \quad \mbox{if $(\mu,
\alpha_i^\vee)=-1$}.  \end{array}\right.\]
If $(\mu,\alpha_i^\vee)=0$, then $n_i^{-1}(z_\mu)=z_\mu$. If $(\mu,
\alpha_i^\vee)=1$, then $(\mu-\alpha_i,\alpha_i^\vee)=-1$ and so 
$n_i^{-1}(z_{\mu-\alpha_i})=-z_\mu$. If $(\mu,\alpha_i^\vee)=-1$,
then $(\mu+\alpha_i,\alpha_i^\vee)=1$ and so $n_i^{-1}(z_{\mu+\alpha_i})=
z_\mu$. Hence, we obtain $n_i^{-1}h_j n_i(z_\mu)=(s_i(\mu),\alpha_j^\vee)
z_\mu= h_j(z_\mu)-(\alpha_i, \alpha_j^\vee)h_i(z_\mu)$
and so $n_i^{-1}h_j n_i=h_j-(\alpha_i,\alpha_j^\vee)h_i=s_i^*(h_j)$;
see Remark~\ref{tract}.

(b) Let $x\in \fg_\alpha$; then $[h_j,x]=\dalp(h_j)x$ 
for all $j\in I$. Then, using (a), we obtain
\[[h_j,n_ixn_i^{-1}]=n_i[s_i^*(h_j),x]n_i^{-1}=\dalp(s_i^*(h_j))
n_ixn_i^{-1}=(s_i(\dalp))(h_j)n_ixn_i^{-1}.\]
Since $n_i\fg n_i^{-1}\subseteq \fg$, we conclude that $n_ixn_i^{-1}\in 
\fg_{s_i(\dalp)}=\fg_{s_i(\alpha)}$.
\end{proof}

Let $k\geq 0$ and $i,i_1\ldots, i_k\in I$. Let $\alpha:=s_{i_1}
\cdots s_{i_k}(\alpha_i)\in \Phi$. Then Lemma~\ref{lemni1} and a simple
induction on $k$ show that
\[n_{i_1}\dots n_{i_k}e_in_{i_k}^{-1}\cdots n_{i_1}^{-1}\in\fg_\alpha.\]
If we denote this element by $\be_\alpha$, then the formula in 
Definition~\ref{defj} translates to 
\[ \be_\alpha(z_\mu)=\left\{\begin{array}{cl} \pm z_{\mu+\alpha} & \quad 
\mbox{if $(\mu,\alpha^\vee)=-1$},\\0 & \quad \mbox{otherwise}.
\end{array}\right.\]
(Indeed, using Lemma~\ref{lemni}, we obtain 
$n_{i_k}^{-1}\cdots n_{i_1}^{-1}(z_\mu)=\pm z_{s_{i_k}\cdots s_{i_1}
(\mu)}=\pm z_{w^{-1}(\mu)}$.
Now note that $(\mu,\alpha^\vee)=(\mu,(w(\alpha_i))^\vee)=(\mu,
w(\alpha_i^\vee))=(w^{-1}, \alpha_i^\vee)$; see Remark~\ref{rem00}. 
Hence, if $(\mu,\alpha^\vee) \neq -1$, then $(w^{-1}(\mu),
\alpha_i^\vee)\neq -1$ and so $e_i(z_{w^{-1} (\mu)})=0$, which implies 
that $\be_\alpha(z_\mu)=0$. On the other hand, if $(\mu,\alpha^\vee)=-1$, 
then $(w^{-1}(\mu),\alpha_i^\vee)=-1$ and so $e_i(z_{w^{-1}(\mu)})=
z_{w^{-1}(\mu)+\alpha_i}$, which implies that
\begin{align*}
\be_\alpha(z_\mu)&=\pm n_{i_1}\cdots n_{i_k}(z_{w^{-1}(\mu)+\alpha_i})=
\pm z_{s_{i_1}\cdots s_{i_k}w^{-1}(\mu)+s_{i_1}\cdots s_{i_k}(\alpha_i)}
=\pm z_{\mu+\alpha},
\end{align*}
as required.) We note that, in each row and in each column of 
$\be_\alpha$, there is at most one non-zero entry (which then is $\pm 1$).
Since $\dim \fg_\alpha=1$, we conclude that $\be_\alpha$ is well-defined
up to a sign, that is, if $l\geq 0$ and $j,j_1,\ldots,j_l\in I$ are also 
such that $\alpha=s_{j_1} \cdots s_{j_l}(\alpha_j)$, then
$n_{i_1}\dots n_{i_k}e_in_{i_k}^{-1}\cdots n_{i_1}^{-1} =\pm 
n_{j_1}\dots n_{j_l}e_jn_{j_l}^{-1}\cdots n_{j_1}^{-1}$.

\begin{defn} \label{choose} Let us choose, for each $\alpha\in\Phi$,
a sequence $i,i_1,\ldots,i_k\in I$ as above such that 
$\alpha=s_{i_1} \cdots s_{i_k}(\alpha_i)\in \Phi$, and set 
\[\be_\alpha:=n_{i_1}\dots n_{i_k}e_in_{i_k}^{-1}\cdots
n_{i_1}^{-1}\in\fg_\alpha.\]
(As discussed, $\be_\alpha$ is well-defined up to a sign; in order to fix
these signs, one could use the ``canonical'' 
Chevalley bases in \cite[\S 5]{my}.) Since $e_i^2=0$ and $\be_\alpha$ is
conjugate to~$e_i$, we have $\be_\alpha^2=0$. Also note that 
$\be_\alpha=\pm e_j$, $\be_{-\alpha}=\pm f_j$ if $\alpha= \alpha_j$ with 
$j\in I$. 
\end{defn}

The rather explicit form of the elements $\be_\alpha$ allows us to 
determine some relations among them, at least up to a sign. The following 
result will be useful in the proof of Chevalley's commutator relations 
in Section~\ref{sec3}.

\begin{prop} \label{comm1} Let $\alpha,\beta\in \Phi$, $\beta\neq\pm 
\alpha$. Then the following hold.
\begin{itemize}
\item[{\rm (a)}] If $\alpha+\beta\not\in\Phi$, then $[\be_\alpha,
\be_\beta]=\be_\beta \be_\alpha \be_\beta=0$. 
\item[{\rm (b)}] If $\alpha+\beta\in\Phi$, then $[\be_\alpha,
\be_\beta]=c\,\be_{\alpha+\beta}$ where $c\in\{\pm 1,\pm 2\}$; we have
$c=\pm 2$ if and only if $\alpha-\beta\in\Phi$.
\item[{\rm (c)}] If $\alpha+\beta\in\Phi$ and $2\alpha+\beta\not\in
\Phi$, then the pairwise products of $\be_\alpha$, $[\be_\alpha,\be_\beta]$
and $\be_\beta \be_\alpha \be_\beta$ are all zero; furthermore, there is a
sign $c'=\pm 1$ such that 
\[\be_\beta \be_\alpha \be_\beta=\left\{\begin{array}{cl} 
c'\be_{\alpha+2\beta} & \quad \mbox{if $\alpha+2\beta\in\Phi$},\\
0 & \quad \mbox{otherwise}.\end{array}\right.\]
\end{itemize}
\end{prop}

\begin{proof} Recall from the general theory of Lie algebras that, for
any $\mu,\nu\in\fh^*$, we have $[\fg_\mu,\fg_\mu]\subseteq \fg_{\mu+\nu}$. 
As in the proof of Lemma~\ref{lemj}, a straightforward computation shows the
following relations (where one uses that $\be_\alpha^2=\be_\beta^2=0$): 
\begin{equation*}
[\be_\alpha,[\be_\alpha,\be_\beta]]=-2\be_\alpha \be_\beta\be_\alpha
\qquad\mbox{and}\qquad [\be_\beta,[\be_\beta,\be_\alpha]]=-2\be_\beta 
\be_\alpha \be_\beta.\tag{$\dagger$}
\end{equation*}
Thus, $\be_\alpha \be_\beta\be_\alpha \in\fg_{2\alpha+\beta}$ if
$2\alpha+\beta\in\Phi$, and $\be_\alpha \be_\beta\be_\alpha=0$ if 
$2\alpha+\beta\not\in\Phi$; an analogous statement holds for 
$\be_\beta \be_\alpha \be_\beta$.

(a) If $\alpha+\beta\not\in\Phi$, then $[\fg_\alpha,\fg_\beta]=\{0\}$
and so $[\be_\alpha,\be_\beta]=0$, $\be_\beta \be_\alpha \be_\beta=0$
(see ($\dagger$)). 

(b) If $\alpha+\beta\in\Phi$, then $[\fg_\alpha,\fg_\beta]=
\fg_{\alpha+ \beta}$ (see \cite[\S 8.4]{H}) and so $[\be_\alpha,
\be_\beta]=c\be_{\alpha+\beta}$ for some $0\neq c\in \C$. Now, we have 
seen that $\be_\alpha,\be_\beta,\be_{\alpha+\beta}$ are represented by 
matrices with all entries in $\{0,\pm 1\}$, with at most one non-zero 
entry in each row and each column. This certainly implies that 
$[\be_\alpha,\be_\beta]=\be_\alpha\be_\beta-\be_\beta\be_\alpha$ is 
represented by a matrix with all entries in $\{0,\pm 1,\pm 2\}$, hence 
$c\in\{\pm 1,\pm 2\}$; furthermore, $c=\pm 2$ precisely when some
entry of $[\be_\alpha,\be_\beta]$ equals $\pm 2$. The latter condition
can only happen if there is some $\mu\in\Psi$ such that 
$\be_\alpha\be_\beta(z_\mu)=-\be_\alpha\be_\beta(z_\mu)\neq 0$. Assume 
that this is the case. Now, we have 
\begin{align*}
\be_\alpha\be_\beta(z_\mu) &=\left\{\begin{array}{cl} 
\pm z_{\mu+\alpha+\beta} & \quad \mbox{if $(\mu,\beta^\vee)=-1$ and
$(\mu+\beta,\alpha^\vee)=-1$},\\ 0 & \quad \mbox{otherwise};
\end{array}\right.\\
\be_\beta\be_\alpha(z_\mu) &=\left\{\begin{array}{cl} 
\pm z_{\mu+\alpha+\beta} & \quad \mbox{if $(\mu,\alpha^\vee)=-1$ and
$(\mu+\alpha,\beta^\vee)=-1$},\\ 0 & \quad \mbox{otherwise}.
\end{array}\right.
\end{align*}
Hence, in particular, we must have $(\mu,\beta^\vee)=-1$ and 
$(\mu+\alpha,\beta^\vee)=-1$, which implies that $(\alpha,\beta)=0$.
Conversely, assume that $(\alpha,\beta)=0$. Let $\mu\in\Psi$ be
such that $\be_{\alpha+\beta}(z_\mu)\neq 0$. Then $(\mu,(\alpha+
\beta)^\vee)=-1$. Now $(\alpha+\beta,\alpha+\beta)=(\alpha,\alpha)+(\beta,
\beta)$. Since there are only two possible root lengths in $\Phi$,
we deduce that $(\alpha,\alpha)=(\beta,\beta)$ and $(\alpha+\beta)^\vee=
\frac{1}{2}(\alpha^\vee+\beta^\vee)$. Since $(\mu,\alpha^\vee)$ and
$(\mu,\beta^\vee)$ are in $\{0,\pm 1\}$, we must have $(\mu,\alpha^\vee)
=(\mu,\beta^\vee)=-1$. Hence, $\be_\alpha\be_\beta(z_\mu)=\pm z_{\mu+
\alpha+\beta}$ and $\be_\beta\be_\alpha(z_\mu)=\pm z_{\mu+ \alpha+\beta}$. 
Since $[\be_\alpha,\be_\beta](z_\mu)\neq 0$, we conclude that 
$[\be_\alpha,\be_\beta](z_\mu)=\pm 2z_{\mu+\alpha+\beta}$ and so $c=\pm 2$.
Thus, we have shown that $c=\pm 2$ if and only if $(\alpha,\beta)=0$.
Finally, since $\alpha+\beta\in\Phi$, the latter condition is equivalent
to the condition that $\alpha-\beta\in\Phi$; see \cite[p.~45]{H}.

(c) We have $\be_\alpha \be_\beta \be_\alpha=0$ since $2\alpha+\beta
\not\in\Phi$ (see ($\dagger$)). This immediately implies the statement about 
the pairwise products of $\be_\alpha$, $[\be_\alpha,\be_\beta]$ and 
$\be_\beta \be_\alpha \be_\beta$. Indeed, we have $\be_\alpha [\be_\alpha,
\be_\beta]= \be_\alpha(\be_\alpha \be_\beta-\be_\beta\be_\alpha)=-\be_\alpha
\be_\beta\be_\alpha=0$; similarly, $[\be_\alpha,\be_\beta]\be_\alpha=0$. 
Furthermore, the products of $\be_\beta \be_\alpha \be_\beta$ with
$\be_\alpha$ and with $[\be_\alpha,\be_\beta]$ will be zero, since 
these products always involve one of the terms $\be_\beta^2=0$ or
$\be_\alpha\be_\beta \be_\alpha=0$.

Finally, we have $\be_\beta \be_\alpha \be_\beta=0$ if 
$\alpha+2\beta\not\in\Phi$. On the other hand, if $\alpha+ 2\beta\in\Phi$, 
then $[\fg_\beta,[\fg_\beta,\fg_\alpha]]=[\fg_\beta,\fg_{\alpha+\beta}]=
\fg_{\alpha+2\beta}$ (see \cite[\S 8.4]{H}) and so $\be_\beta\be_\alpha
\be_\beta=c'\be_{\alpha+ 2\beta}$ for some $0\neq c' \in \C$. Again, since 
$\be_\alpha,\be_\beta,\be_{\alpha+2\beta}$ are represented by matrices 
with all entries in $\{0,\pm 1\}$, with at most one non-zero entry in each 
row and each column, the same is true for the product $\be_\beta \be_\alpha
\be_\beta$ and so $c'=\pm 1$.
\end{proof}

\section{Chevalley groups via minuscule weights} \label{sec3}

Let $\Psi\subseteq \Lambda$ be as in Definition~\ref{defj} and 
consider the corresponding Lie algebra $\fg\subseteq \gl(M)$. We 
will now pass from $\C$ to an arbitrary commutative ring $R$ with~$1$.

Let $\bar{M}$ be a free $R$-module with a basis $\{\bar{z}_\mu\mid \mu
\in\Psi\}$. We define $R$-linear maps $\bar{e}_i\colon \bar{M}\rightarrow 
\bar{M}$ and $\bar{f}_i\colon \bar{M}\rightarrow \bar{M}$ by analogous 
formulae as in Definition~\ref{defj}; again, we have $\bar{e}_i^2=
\bar{f}_i^2=0$ for all $i\in I$. For any $t\in R$ we set 
\[ x_i(t):=\mbox{id}_{\bar{M}}+t\bar{e}_i \qquad \mbox{and}\qquad 
y_i(t):=\mbox{id}_{\bar{M}}+t\bar{f}_i.\]
Note that $x_i(t+t')=x_i(t)x_i(t')$ and $y_i (t+t')=y_i(t)y_i(t')$ for 
all $t,t'\in R$ (since $\bar{e}_i^2=\bar{f}_i^2=0$). Furthermore, 
$x_i(0)=y_i(0)=\mbox{id}_{\bar{M}}$. Hence, $x_i(t)$ and $y_i(t)$ are 
invertible where $x_i(t)^{-1}=x_i(-t)$ and $y_i(t)^{-1}=y_i(-t)$. So we 
obtain a group 
\[G=G_R(\Psi):=\langle x_i(t),y_i(t)\mid i\in I,t\in R \rangle\subseteq 
\GL(\bar{M}),\]
which we call the {\em Chevalley group} of type $\Psi$ over $R$. Our first 
aim is to exhibit subgroups in $G$ which form the ingredients of a 
split $BN$-pair (as in \cite[Chap.~8]{Ca1}).

For $i\in I$ and $t\in R^\times$, we set $\bar{n}_i(t):=x_i(t)y_i(-t^{-1})
x_i(t)\in G$. By the same computations as in the proof of 
Lemma~\ref{lemni}, we find that 
\[ \bar{n}_i(t)(\bar{z}_\mu)=\left\{\begin{array}{cl} \bar{z}_\mu & \quad 
\mbox{if $(\mu, \alpha_i^\vee)=0$},\\ -t^{-1}\bar{z}_{\mu-\alpha_i} & 
\quad \mbox{if $(\mu,\alpha_i^\vee)=1$},\\ t\bar{z}_{\mu+\alpha_i} & 
\quad \mbox{if $(\mu, \alpha_i^\vee)=-1$}.  \end{array}\right.\]
In particular, each $\bar{n}_i(t)$ is represented by a monomial matrix; 
furthermore, setting $\bar{n}_i:=\bar{n}_i(1)$, we have $\bar{n}_i^2
(\bar{z}_\mu)=(-1)^{(\mu,\alpha_i^\vee)}\bar{z}_\mu$ and $\bar{n}_i^4=
\operatorname{id}_{\bar{M}}$. 

\begin{rem} \label{reduce} (a) Let $M_\Z:=\langle z_\mu\mid \mu\in\Psi
\rangle_\Z\subseteq M$. By the formulae in Definition~\ref{defj} it is 
obvious that $e_i(M_\Z)\subseteq M_\Z$ and $f_i(M_\Z)\subseteq M_\Z$ for 
all~$i$. Hence we also have $n_i=(\mbox{id}_M+e_i)(\mbox{id}_M-f_i)
(\mbox{id}_M+e_i)\in\GL(M_\Z)$. Consequently, the elements $\be_\alpha$ 
($\alpha\in\Phi$) in Definition~\ref{choose} satisfy $\be_\alpha(M_\Z) 
\subseteq M_\Z$. Now, we can naturally identify $\bar{M}=R\otimes_\Z M_\Z$. 
Then $\bar{e}_i$, $\bar{f}_i$, $\bar{n}_i$ are the maps induced by $e_i$, 
$f_i$, $n_i$.

(b) Let $\alpha\in\Phi$ and $\bar{\be}_\alpha\colon\bar{M}
\rightarrow\bar{M}$ be the map induced by $\be_\alpha$; then 
$\bar{\be}_\alpha^2=0$ and 
\[ \bar{\be}_\alpha(\bar{z}_\mu)=\left\{\begin{array}{cl}
\pm \bar{z}_{\mu+\alpha} & \quad \mbox{if $(\mu,\alpha^\vee)=-1$},\\
0 & \quad \mbox{otherwise}.\end{array}\right.\]
We define $x_\alpha(t):=\mbox{id}_{\bar{M}}+t\bar{\be}_\alpha$
for $t\in R$. Then $x_\alpha(t)\in\GL(\bar{M})$ where
$x_\alpha(t)^{-1}=x_\alpha(-t)$ for all $t\in R$; furthermore,
$x_\alpha(t+t')=x_\alpha(t)x_\alpha(t')$ for all $t,t' \in R$. We 
have $x_\alpha(t)\in G$ since $\be_\alpha$ is obtained by conjugating
a suitable $e_i$ by a product of various $n_j$ and, hence, an analogous 
statement is true for $\bar{\be}_\alpha$ as well. Since $e_i=\pm 
\be_{\alpha_i}$ and $f_i=\pm \be_{-\alpha_i}$ for all $i\in I$
(see Definition~\ref{choose}), we have 
\[G=G_R(\Psi)=\langle x_\alpha(t)\mid \alpha\in\Phi,t\in R\rangle.\]
We can now define subgroups of $G$ as follows:
\[U^+:=\langle x_\alpha(t)\mid \alpha\in\Phi^+, t\in R\rangle 
\qquad \mbox{and}\qquad U^-:=\langle x_\alpha(t)\mid
\alpha\in\Phi^-, t\in R\rangle.\]
(Note that these do not depend on the choice of the elements $\be_\alpha$
in Definition~\ref{choose}.) Now let us choose an enumeration of the 
elements of $\Psi$ as in Remark~\ref{rem10}. Then $\bar{\be}_\alpha$ will 
be represented by a strictly upper triangular matrix if $\alpha\in\Phi^+$, 
and by a strictly lower triangular matrix if $\alpha\in\Phi^-$. Consequently,
\begin{align*}
U^+ & \mbox{ consists of upper triangular matrices with $1$ on the 
diagonal},\\ U^- & \mbox{ consists of lower triangular matrices with 
$1$ on the diagonal}.
\end{align*}
We can now establish Chevalley's commutator relations \cite{Ch} in 
our setting.
\end{rem}

\begin{prop} \label{comm} Let $\alpha, \beta \in \Phi^+$, $\beta\neq 
\pm\alpha$, and $t,u\in R$. Then $x_\alpha(t)x_\beta(u)=x_\beta(u)
x_\alpha(t)$ if $\alpha+\beta\not\in \Phi$. Now assume that $\alpha+
\beta \in \Phi$ and let $c\in\{\pm 1,\pm 2\}$ be such that $[\be_\alpha,
\be_\beta]=c\,\be_{\alpha+\beta}$, as in Proposition~\ref{comm1}. Then
the following hold.
\begin{itemize}
\item[(a)] If $2\alpha+\beta \not\in \Phi$ and $\alpha+2\beta\not\in\Phi$, 
then $x_\beta(-u)x_\alpha(t)x_\beta(u)=x_\alpha(t)x_{\alpha+\beta}(c tu)$.
\item[(b)] If $2\alpha+\beta \not\in \Phi$ and $\alpha+2\beta\in\Phi$, 
then 
\[x_\beta(-u)x_\alpha(t)x_\beta(u)=x_\alpha(t)x_{\alpha+\beta}(ctu) 
x_{\alpha+2\beta}(-c'tu^2)\]
where the factors on the right hand side commute with each other and  
$c'=\pm 1$ is determined by the relation $\be_\beta\be_\alpha
\be_\beta=c'\be_{\alpha+2\beta}$ (as in Proposition~\ref{comm1}).
\item[(c)] If $2\alpha+\beta \in \Phi$, then $\alpha+2\beta\not\in\Phi$ and
\[x_\beta(-u)x_\alpha(t)x_\beta(u)=x_\alpha(t)x_{\alpha+\beta}(ctu) 
x_{2\alpha+\beta}(c''t^2u),\]
where $c''=\pm 1$ is determined by the relation $\be_\alpha\be_\beta
\be_\alpha=c''\be_{2\alpha+\beta}$.
\end{itemize}
\end{prop}

\begin{proof} (Cf.\ \cite[\S 5.2]{Ca1}.) First note that the relations
in Proposition~\ref{comm1} also hold when we replace each $\be_{\gamma}
\colon M\rightarrow M$ (for $\gamma\in\Phi$) by the corresponding 
$\bar{\be}_{\gamma}\colon \bar{M}\rightarrow \bar{M}$. Now, we have 
$x_\alpha(t)=\mbox{id}_{\bar{M}}+t\bar{\be}_\alpha$, $x_\beta(\pm u)=
\mbox{id}_{\bar{M}} \pm u\bar{\be}_\beta$ and so 
\begin{align*}
x_\beta(-u)&x_\alpha(t)x_\beta(u)=(\mbox{id}_{\bar{M}}-
u\bar{\be}_\beta)(\mbox{id}_{\bar{M}}+ t\bar{\be}_\alpha+u\bar{\be}_\beta+
tu\bar{\be}_\alpha \bar{\be}_\beta)\\ &= \mbox{id}_{\bar{M}}+t
\bar{\be}_\alpha +u\bar{\be}_\beta+ tu\bar{\be}_\alpha \bar{\be}_\beta-
u\bar{\be}_\beta- tu\bar{\be}_\beta\bar{\be}_\alpha - u^2\bar{\be}_\beta^2-
tu^2\bar{\be}_\beta\bar{\be}_\alpha\bar{\be}_\beta\\&=\mbox{id}_{\bar{M}}+
t\bar{\be}_\alpha+tu[\bar{\be}_\alpha,\bar{\be}_\beta]-tu^2\bar{\be}_\beta
\bar{\be}_\alpha \bar{\be}_\beta.
\end{align*}
If $\alpha+\beta\not\in\Phi$, then $[\bar{\be}_\alpha,\bar{\be}_\beta]=
\bar{\be}_\beta \bar{\be}_\alpha\bar{\be}_\beta=0$ by 
Proposition~\ref{comm1}(a). Consequently, we have $x_\beta(-u)
x_\alpha(t)x_\beta(u)=x_\alpha(t)$ in this case, as required. Now 
assume that $\alpha+\beta\in\Phi$ and let $c\in\{\pm 1,\pm 2\}$ be such
that $[\bar{\be}_\alpha, \bar{\be}_\beta]=c\,\bar{\be}_{\alpha+ \beta}$.
Since the diagram of our root system is as in Table~\ref{Mcoxgraphs}, one 
easily sees that either $2\alpha+\beta$ or $\alpha+2\beta$ is a root, but 
not both. This leads to the three cases (a), (b), (c).

Now, if $2\alpha+\beta\not \in\Phi$, then we obtain using 
Proposition~\ref{comm1}(c):
\begin{align*}
x_\beta(-u)x_\alpha(t)x_\beta(u)&=
(\mbox{id}_{\bar{M}}+t\bar{\be}_\alpha)(\mbox{id}_{\bar{M}}+tu
[\bar{\be}_\alpha,\bar{\be}_\beta])(\mbox{id}_{\bar{M}}-tu^2
\bar{\be}_\beta\bar{\be}_\alpha \bar{\be}_\beta)\\
&=(\mbox{id}_{\bar{M}}+t\bar{\be}_\alpha)(\mbox{id}_{\bar{M}}+ctu
\bar{\be}_{\alpha+\beta})(\mbox{id}_{\bar{M}}-tu^2\bar{\be}_\beta
\bar{\be}_\alpha \bar{\be}_\beta)\\ &=x_\alpha(t)x_{\alpha+\beta}(ctu)
(\mbox{id}_{\bar{M}}-tu^2\bar{\be}_\beta\bar{\be}_\alpha\bar{\be}_\beta);
\end{align*}
here, $\bar{\be}_\beta\bar{\be}_\alpha\bar{\be}_\alpha=0$ if $\alpha+
2\beta\not\in\Phi$, and $\bar{\be}_\beta\bar{\be}_\alpha \bar{\be}_\alpha=
c'\bar{\be}_{\alpha+2\beta}$ where $c'=\pm 1$, otherwise. This yields the
formulae in (a) and (b). On the other hand, if $2\alpha+\beta \in \Phi$, 
then $\alpha+2\beta\not\in \Phi$ and so the previous argument (exchanging 
the roles of $\alpha, \beta$) yields that 
\[x_\alpha(-t)x_\beta(u)x_\alpha(t)=x_\beta(u)x_{\alpha+\beta}(-ctu) 
x_{2\alpha+\beta}(-c''t^2u),\]
where $c''=\pm 1$ 
satisfies $\be_\alpha\be_\beta \be_\alpha=c''\be_{2\alpha+\beta}$. First 
multiplying this identity on the left by $x_\beta(-u)$, then taking the 
inverse of both sides and finally multiplying on the left by $x_\alpha(t)$
yields the desired formula for $x_\beta(-u)x_\alpha(t) x_\beta(u)$ in (c).
\end{proof}

\begin{cor} \label{subgru} For $\alpha\in\Phi$ let $X_\alpha:=
\{x_\alpha (t)\mid t\in R\}\subseteq G$.
\begin{itemize}
\item[(a)] We have $U^+=\prod_{\alpha \in\Phi^+} X_\alpha$ where the 
product is taken in some fixed order.
\item[(b)] Let $i\in I$. Then $U_i^+:=\langle X_\alpha\mid \alpha_i
\neq \alpha \in\Phi^+\rangle\subseteq U$ is a normal subgroup and 
$U=X_{\alpha_i}.U_i^+$; furthermore, $\bar{n}_iX_{\alpha_i}\bar{n}_i^{-1}=
X_{-\alpha_i}$ and $\bar{n}_iU_i^+\bar{n}_i^{-1}=U_i^+$.
\end{itemize}
\end{cor}

\begin{proof} The commutator relations imply (a) and the first two 
statements in (b), by a purely group-theoretical argument; cf.\ 
\cite[\S 5.3]{Ca1}. Now, $\bar{n}_i \bar{\be}_\alpha \bar{n}_i^{-1}=
\pm \bar{\be}_{s_i(\alpha)}$ for all $\alpha\in\Phi$ (by the discussion
following Lemma~\ref{lemni1}). Hence, $\bar{n}_iX_\alpha \bar{n}_i^{-1}=
X_{s_i(\alpha)}$ and $\bar{n}_iX_{\alpha_i}\bar{n}_i^{-1}=X_{-\alpha_i}$. 
Finally, it is well-known that $s_i(\alpha) \in\Phi^+$ 
for all $\alpha\in\Phi^+$ such that $\alpha\neq \alpha_i$; see 
\cite[2.1.5]{Ca1}. Hence, $n_i U_i^+ n_i^{-1}=U_i^+$. 
\end{proof}

The next step is to define a diagonal subgroup (or ``torus'') in $G$.

\begin{lem} \label{lem22} Let $i\in I$, $t\in R^\times$ and set $h_i(t)
:=\bar{n}_i(t)\bar{n}_i(-1)\in G$. Then 
\[ h_i(t)(\bar{z}_\mu)=t^{(\mu,\alpha_i^\vee)}\bar{z}_\mu.\]
Consequently, $h_i(1)=1$ and $h_i(tt')=h_i(t)h_i(t')$ for all $t,t'\in
R^\times$; thus, 
\[H:=\langle h_i(t)\mid i\in I,t\in R^\times\rangle \subseteq G\]
is an abelian group all of whose elements are represented by diagonal 
matrices. Given $t_i\in R^\times$ ($i \in I$), we have 
$\prod_{i\in I} h_i(t_i)=1\Leftrightarrow \prod_{i\in I} t_i^{(\mu,
\alpha_i^\vee)}=1$  for all $\mu\in \Psi$.
\end{lem}

\begin{proof} Using the above formulae for the action of 
$\bar{n}_i(t)$, we obtain 
\[h_i(t)(\bar{z}_\mu)=\bar{n}_i(t)\bigl(\bar{n}_i(-1)(\bar{z}_\mu)\bigr)
=\left\{\begin{array}{cl} \bar{n}_i(t)(\bar{z}_\mu) & \quad \mbox{if $(\mu, 
\alpha_i^\vee)=0$},\\ \bar{n}_i(t)(\bar{z}_{\mu-\alpha_i}) & \quad 
\mbox{if $(\mu, \alpha_i^\vee)=1$},\\ -\bar{n}_i(t)(\bar{z}_{\mu+\alpha_i}) & 
\quad \mbox{if $(\mu,\alpha_i^\vee)=-1$}. \end{array}\right.\]
Now, if $(\mu, \alpha_i^\vee)=0$, then $\bar{n}_i(t)(\bar{z}_\mu)=
\bar{z}_\mu$. If $(\mu, \alpha_i^\vee)=1$, then $(\mu-\alpha_i,
\alpha_i^\vee)=-1$ and so $\bar{n}_i(t)(\bar{z}_{\mu-\alpha_i})=
t\bar{z}_\mu$. Finally, if $(\mu,\alpha_i^\vee)=-1$, then $(\mu+\alpha_i,
\alpha_i^\vee)=1$ and so $\bar{n}_i(t) (\bar{z}_{\mu-\alpha_i})=
-t^{-1}\bar{z}_\mu$. This yields the desired formula. 

The last statement about $\prod_{i\in I} h_i(t_i)$ is then clear. 
\end{proof}

\begin{lem} \label{lem23} Let $N:=\langle \bar{n}_i(t)\mid i\in I,
t\in R^\times \rangle \subseteq G$. Then $N=\langle H,\bar{n}_i\;
(i\in I)\rangle$ and $H$ is a normal subgroup of $N$. We have 
$\bar{n}_i\not\in H$ and $\bar{n}_i^2\in H$ for all $i\in I$. Furthermore,
$N\cap U^+=N\cap U^-=\{1\}$.
\end{lem}

\begin{proof} By Lemma~\ref{lem22}, we have $\bar{n}_i(-1)=
\bar{n}_i^{-1}h_i(1)=\bar{n}_i^{-1}$ and so $h_i(t)=\bar{n}_i(t)
\bar{n}_i^{-1}$ for all $i\in I$ and $t\in R^\times$. This shows that $N=
\langle H,\bar{n}_i \;(i\in I)\rangle$. Now let $i\in I$ be fixed. We 
have $\bar{n}_i^2= h_i(-1)\in H$. By Lemma~\ref{rem00a}, there exists 
some $\mu\in\Psi$ such that $(\mu, \alpha_i^\vee)\neq 0$. But then $(\mu,
\alpha_i^\vee)=\pm 1$ and so $s_i(\mu)=\mu\pm \alpha_i$. Thus, $\bar{n}_i
(\bar{z}_\mu)=\pm \bar{z}_{\mu \pm \alpha_i}$ and so $\bar{n}_i$
is not represented by a diagonal matrix. Hence, $\bar{n}_i \not\in H$. 

It remains to show that $H$ is a normal subgroup of $N$. 
Let $j\in I$, $t\in R^\times$. Then 
\[\bar{n}_ih_j(t)\bar{n}_i^{-1}(z_\mu)=\left\{\begin{array}{cl} 
t^{(\mu,\alpha_j^\vee)} \bar{n}_i(\bar{z}_\mu) & \quad \mbox{ if $(\mu,
\alpha_i^\vee)=0$},\\ t^{(\mu-\alpha_i,\alpha_j^\vee)}\bar{n}_i
(\bar{z}_{\mu-\alpha_i}) & \quad \mbox{ if $(\mu,\alpha_i^\vee)=1$},\\
-t^{(\mu+\alpha_i,\alpha_j^\vee)}\bar{n}_i(\bar{z}_{\mu+\alpha_i}) & \quad 
\mbox{ if $(\mu,\alpha_i^\vee)=-1$}.\end{array}\right.\]
Now, in the first case, we have $\bar{n}_i(\bar{z}_\mu)=\bar{z}_\mu$. In
the second case, $(\mu-\alpha_i,\alpha_i^\vee)=-1$ and so $\bar{n}_i
(\bar{z}_{\mu-\alpha_i})=\bar{z}_\mu$; also note that $\mu-\alpha_i=s_i
(\alpha)$. In the third case, $(\mu+\alpha_i,\alpha_i^\vee)=1$ and so 
$\bar{n}_i(\bar{z}_{\mu-\alpha_i})=-\bar{z}_\mu$; also note that 
$\mu+\alpha_i= s_i(\alpha)$. Hence,
\[\bar{n}_ih_j(t)\bar{n}_i^{-1}(z_\mu)=t^{(s_i(\mu),\alpha_j^\vee)}
\bar{z}_\mu.\] 
Now, $(s_i(\mu),\alpha_j^\vee)=(\mu,s_i(\alpha_j^\vee))=(\mu,\alpha_j^\vee-
(\alpha_j^\vee,\alpha_i^\vee) \alpha_i)=(\mu,\alpha_j^\vee)-
a_{ji}(\mu,\alpha_i^\vee)$, so 
\[ \bar{n}_ih_j(t)\bar{n}_i^{-1}(\bar{z}_\mu)=t^{(\mu,\alpha_j^\vee-
a_{ji}(\mu, \alpha_i^\vee)}(\bar{z}_\mu)=\bigl(h_j(t)h_i(t)^{-a_{ji}})\bigr)
(\bar{z}_\mu),\] 
which shows that $\bar{n}_ih_j(t)\bar{n}_i^{-1}=h_j(t)h_i(t)^{-a_{ji}} 
\in H$, as claimed. Finally, let $g\in N\cap U^{\pm}$; we want to show
that $g=1$. Now, since $N$ is generated by $H$ and the $\bar{n}_i$ 
($i\in I$), we can write $g=\bar{n}_{i_1}\cdots \bar{n}_{i_r}h$ where 
$i_j\in I$ and $h\in H$. Let $w=s_{i_1} \cdots s_{i_r}\in W$. By 
Lemma~\ref{lem22} and the formulae for the action of the elements 
$\bar{n}_i$, we have $g(\bar{z}_\mu)=c_\mu \bar{z}_{w(\mu)}$ for all $\mu
\in\Psi$, where $c_\mu\in R^\times$. Thus, $g$ is represented by a monomial
matrix. On the other hand, $g\in U^{\pm}$ and so $g$ is represented by 
a triangular matrix with $1$ on the diagonal. Hence, $g=1$.
\end{proof}

\begin{lem} \label{lem24} Let $i\in I$, $\alpha\in\Phi$, $u\in R$ and 
$t\in R^\times$. Then 
\[h_i(t)x_\alpha(u)h_i(t)^{-1}=x_\alpha(ut^{(\alpha,\alpha_i^\vee)}).\]
Consequently, $U^{\pm}$ are normalised by $H$ and, hence, we
obtain subgroups $B^{\pm}:=U^{\pm}.H\subseteq G$. 
We have $B^+\cap B^-=H$ and $B^{\pm}\cap N=H$.
\end{lem}

\begin{proof} We have $h_i(t)^{-1}=h_i(t^{-1})$. Using Lemma~\ref{lem22},
we obtain 
\begin{align*}
h_i(t)x_\alpha(u)&h_i(t)^{-1}(\bar{z}_\mu)=h_i(t)x_\alpha(u)
h_i(t^{-1})(\bar{z}_\mu)=t^{-(\mu,\alpha_i^\vee)}h_i(t)x_\alpha(u)
(\bar{z}_\mu)\\ &=t^{-(\mu,\alpha_i^\vee)}h_i(t)(\bar{z}_\mu+
u\be_\alpha(\bar{z}_\mu))=\bar{z}_\mu+t^{-(\mu,\alpha_i^\vee)}
uh_i(t)(\be_\alpha(\bar{z}_\mu)).
\end{align*}
If $(\mu, \alpha^\vee)\neq -1$, then $\be_\alpha(\bar{z}_\mu)=0$ and 
so $h_i(t)x_\alpha(u)h_i(t)^{-1}(\bar{z}_\mu)=\bar{z}_\mu$. But, in this 
case, we also have $x_\alpha(u')(\bar{z}_\mu)=z_\mu+u'\be_\alpha
(\bar{z}_\mu)=\bar{z}_\mu$ for any $u'\in R$, as required. If $(\mu,
\alpha_j^\vee)=-1$, then $\be_\alpha(\bar{z}_\mu)=\epsilon 
\bar{z}_{\mu+\alpha}$ (with $\epsilon=\pm 1$) and so 
\[h_i(t)x_\alpha(u)h_i(t)^{-1}(\bar{z}_\mu)=\bar{z}_\mu+\epsilon 
t^{-(\mu,\alpha_i^\vee)}ut^{(\mu+\alpha,\alpha_i^\vee)}
(\bar{z}_{\mu+\alpha})=\bar{z}_\mu+\epsilon ut^{(\alpha,\alpha_i^\vee)}
\bar{z}_{\mu+\alpha},\]
which is the same as $x_\alpha(ut^{(\alpha,\alpha_i^\vee)})(\bar{z}_\mu)$. 
By Lemma~\ref{lem23}, we have $U^{\pm}\cap N=\{1\}$ which implies
that $B^+\cap B^-=H$ and $B^{\pm}\cap N=H$.
\end{proof}

\begin{cor} \label{perfect} For $i\in I$ define the subgroup $G_i:=
\langle x_i(t),y_i(t)\mid t\in R \rangle\subseteq G$. If there exists
some $t\in R^\times$ such that $t^2-1\in R^\times$, then $G_i=[G_i,G_i]$.
\end{cor}

\begin{proof} Let $u\in R$. We have $\bar{n}_i\in G$ and, by 
Lemma~\ref{lem22}, we have $h_i(t)\in G$. Furthermore, $h_i(t)x_i(u) 
h_i(t)^{-1}=x_i(ut^2)$ and so $x_i(ut^2-u)=x_i(ut^2)x_i(u)^{-1}\in 
[G_i,G_i]$; see Lemma~\ref{lem24}. Since $t^2-1\in R^\times$, we conlude 
that $x_i(u')\in [G_i,G_i]$ for all $u'\in R$. Similarly, $y_i(u')\in [G_i,
G_i]$ for all $u'\in R$.
\end{proof}

\begin{lem} \label{lemb} There is a unique group isomorphism
$W\stackrel{\sim}{\rightarrow} N/H$, such that $s_i\mapsto \bar{n}_iH$ 
for all $i\in I$.
\end{lem}

\begin{proof} (Cf.\ \cite[Lemma~22, p.~31]{St}.) By Lemma~\ref{lem23}, 
$\bar{n}_i^2\in H$ for all $i\in I$. Furthermore, by 
Proposition~\ref{braidr}, the elements $n_i$ ($i\in I$) satisfy the braid 
relations and, hence, the same is true for the elements $\bar{n}_i$ 
($i\in I$). Now, it is well-known that $W$ has a presentation with 
generators $\{s_i \mid i\in I\}$ subject to the relations $s_i^2=1$
($i\in I$) and the braid relations (for the $s_i$). Thus, we obtain a 
unique group homomorphism $W\rightarrow N/H$ such that $s_i\mapsto 
\bar{n}_iH$ for all $i\in I$. It is surjective by Lemma~\ref{lem23}. To 
prove injectivity, assume that $w\in W$ maps to $1\in N/H$. Write 
$w=s_{i_1}\cdots s_{i_k}$ where $k\geq 0$ and $i_1,\ldots,i_k\in I$. 
Then $\bar{n}_{i_1}\cdots \bar{n}_{i_k}\in H$. Since $\bar{n}_iX_\alpha 
\bar{n}_i^{-1}=X_{s_i(\alpha)}$ for all $i\in I$ and $\alpha\in\Phi$, we 
conclude that $\bar{n}_{i_1}\cdots \bar{n}_{i_k}X_\alpha \bar{n}_{i_k}^{-1}
\cdots \bar{n}_{i_1}=X_{w(\alpha)}$. On the other hand, $H$ normalises 
$X_\alpha$ by Lemma~\ref{lem24} and so $X_{w(\alpha)}=X_\alpha$. Now 
recall that $X_\alpha$ consists of upper triangular matrices if 
$\alpha\in\Phi^+$, and of lower triangular matrices if $\alpha\in\Phi^-$; 
also note that $X_\alpha \neq \{1\}$ (since $\bar{\be}_\alpha\neq 0$). 
Thus, if $\alpha\in\Phi^+$, then the condition $X_{w(\alpha)}=X_\alpha$ 
implies that $w(\alpha)\in\Phi^+$. So we must have $w=1$. 
\end{proof}

Finally, let us assume from now on that $R=k$ is a field. By exactly the 
same arguments as in \cite[\S 8.2]{Ca1}, one sees that the subgroups $B^+$ 
and $N$ form  a $BN$-pair in~$G$. Similarly, $B^-$ and $N$ also form a 
$BN$-pair in~$G$. Since $N_G(B^{\pm})=B^{\pm}$ (see \cite[8.3.3]{Ca1}),
it follows that $Z(G)$ (the center of $G$) is contained in~$B^+\cap B^-=H$. 

\begin{cor}[Cf.\ \protect{\cite[Lemma 28, p.~43]{St}}] \label{center} We have
$|Z(G)|<\infty$ and 
\begin{center}
$Z(G)=\big\{ \prod_{i\in I} h_i(t_i) \in H\,(t_i\in k^\times) \,\big|\, 
\prod_{i\in I} t_i^{(\alpha,\alpha_i^\vee)}=1 \mbox{ for
all $\alpha\in \Phi$}\big\}$.
\end{center}
\end{cor}

\begin{proof} Let $h\in H$ and write $h=\prod_{i\in I} h_i(t_i)$ with
$t_i\in k^\times$. Then $h\in Z(G)\Leftrightarrow hx_\alpha(u)=
x_\alpha(u)h$ for all $\alpha \in \Phi$ and $u\in k$. It remains to 
use Lemma~\ref{lem24}.
\end{proof}

\begin{cor} \label{center1} Let us choose an enumeration of the 
elements of $\Psi$ as in Remark~\ref{rem10}. Then $B^+$ consists precisely
of all elements of $G$ which are represented by upper triangular matrices.
Similarly, $B^-$ consists precisely of all elements of $G$ which are 
represented by lower triangular matrices.
\end{cor} 

\begin{proof} Since $U^+$ is represented by upper triangular matrices 
with~$1$ on the diagonal and $H$ by diagonal matrices, it is clear that
$B^+$ is represented by upper triangular matrices. Conversely, 
assume that $g\in G$ is represented by an upper triangular 
matrix. By the Bruhat decomposition (see \cite[8.2.3]{Ca1}), we can 
write $g=bnb'$ where $b,b' \in B^+$ and $n\in N$. Then $n=b^{-1}gb'^{-1}$ 
is also represented by an upper triangular matrix. Since, on the other 
hand, $n$ is represented by a monomial matrix, we conclude using 
Lemma~\ref{lemb} that $n\in H$ and, hence, $g\in B^+$. 
\end{proof}

As in \cite[\S 5]{St}, the above results have the following
significance to the theory of semisimple algebraic groups. (For basic 
notions about algebraic groups, see \cite{my0}.)

\begin{thm} \label{alggr} Assume that $R=k$ is an algebraically closed
field. Then $G$ is a simple algebraic group, $B^{\pm}
\subseteq G$ are Borel subgroups, $H=B^+\cap B^-$ is a maximal torus
and $N/H\cong W$. The abelian group $\Z\Psi\subseteq \Lambda$ is 
naturally isomorphic to the character group of $H$, where the isomorphism 
is given by sending $\mu\in\Z\Psi$ to the unique homomorphism $\hat{\mu} 
\colon H\rightarrow k^\times$ such that $\hat{\mu}(h_i(t))=t^{(\mu,
\alpha_i^\vee)}$ for $i\in I$, $t\in k^\times$.
\end{thm}

\begin{proof} Exactly as in \cite[\S 5]{St}, one sees that $G$ is
a connected algebraic group; furthermore, $B^{\pm}\subseteq G$ are
closed connected solvable subgroups and $H=B^+\cap B^- \subseteq G$ 
is a torus. Using Lemma~\ref{perfect}, we see that $G=[G,G]$. A 
general criterion about groups with a $BN$-pair now implies 
that every proper normal subgroup of $G$ is contained in $Z(G)$; see
\cite[IV, \S 2, no.~7]{B}. Hence, $G$ is a simple algebraic group.

Using once more Lemma~\ref{perfect} and the argument in
\cite[p.~59]{St}, we see that $B^{\pm}$ are Borel subgroups of $G$.
Since $B^{\pm}=U^{\pm}.H$ and $U^{\pm}\cap H=\{1\}$, it also follows
that $H$ is a maximal torus of $G$. The character group of $H$ is
determined as follows (cf.\ \cite[p.~60]{St}). Choosing an enumeration of 
$\Psi$, we obtain a closed embedding $G\subseteq \GL_d(k)$ where 
$d=|\Psi|$. Under this embedding, $H$ is contained in the maximal
torus $T_d\subseteq \GL_d(k)$ consisting of diagonal matrices. Hence, 
the characters of $H$ (i.e., algebraic homomorphisms $H\rightarrow k^\times$)
are obtained by restriction from the characters of $T_d$; see, e.g.,  
\cite[\S 3.1]{my0}. But a basis of the character group of $T_d$ is 
simply given by the $d$ homomorphisms $\chi_i\colon T_n\rightarrow
k^\times$ ($1\leq i \leq d$), where $\chi_i$ sends a diagonal matrix to
its $i$th diagonal entry. It remains to note that the restrictions of 
these characters $\chi_i$ to $H$ are just the maps $\hat{\mu}\colon
H\rightarrow k^\times$ for $\mu\in \Psi$ (see Lemma~\ref{lem22}). Thus, 
we obtain a surjective homomorphism of abelian groups $\Z\Psi \rightarrow 
\Hom_{\text{alg}}(H, k^\times)$, $\mu\mapsto \hat{\mu}$, and it is easy 
to see that this map is also injective.
\end{proof}

\section{Remarks and examples} \label{sec4}

As mentioned in the introduction, the point of the above construction of 
the Chevalley group $G=G_R(\Psi)$ is that it does not involve any choices 
of signs and that it is completely explicit; in particular, it can be 
easily implemented on a computer. We begin by writing down the recipe for 
doing this.

\begin{rem} \label{algo} Let us fix one of the Dynkin diagrams in 
Table~\ref{Mcoxgraphs} and let $A=(a_{ij})_{i,j\in I}$ be the 
corresponding Cartan matrix. 

\medskip
{\em Step 1}. We identify the weight lattice $\Lambda$ with the free 
$\Z$-module of all $I$-tuples $(v_i)_{i \in I}$ where $v_i \in \Z$ for 
$i \in I$. Under this identification, $\varpi_i$ is the $I$-tuple with 
$1$ at position~$i$, and $0$ everywhere else. For $j \in I$, let 
\begin{center}
$\alpha_j :=(a_{ij})_{i \in I}=\sum_{i \in I} a_{ij}\varpi_i \in \Lambda$
\end{center}
and define $\sigma_j \in \GL_I(\Z)$ by $\sigma_j(\varpi_i)=\varpi_i- 
\delta_{ij} \alpha_j$ for all $i \in I$. (Note that $\sigma_j^2=
\mbox{id}_\Lambda$.)

\medskip
{\em Step 2}. Let $i_0\in I$ be a node marked ``$\circ$'' in our given 
Dynkin diagram. Then let $\Psi_{i_0}\subseteq \Lambda$ be the orbit of 
$\varpi_{i_0}$ under the action of the subgroup $\langle \sigma_j \mid j 
\in I\rangle\subseteq \GL_I(\Z)$. In Table~\ref{Morbits}, we list the 
sizes of these orbits for the various cases. 

\medskip
{\em Step 3}. Let $\Psi\subseteq \Lambda$ be a non-empty union of orbits 
$\Psi_{i_0}$ as in Step~2. Let $R$ be any commutative ring with $1$. Let 
$\bar{M}$ be a free $R$-module with a basis $\{\bar{z}_\mu\mid\mu\in\Psi\}$. 
For $i \in I$ and $t \in R$, we define $R$-linear maps $x_i(t)\colon \bar{M}
\rightarrow \bar{M}$ and $y_i(t)\colon \bar{M}\rightarrow \bar{M}$ by 
the following formulae:
\begin{align*}
x_i(t)\colon \quad\bar{z}_\mu&\mapsto \left\{\begin{array}{cl} \bar{z}_\mu
+ t\bar{z}_{\mu+\alpha_i} & \quad \mbox{if $\mu+\alpha_i\in \Psi$},\\
\bar{z}_\mu & \quad \mbox{otherwise}, \end{array}\right.\\
y_i(t)\colon \quad\bar{z}_\mu&\mapsto \left\{\begin{array}{cl} 
\bar{z}_\mu+t\bar{z}_{\mu-\alpha_i} & \quad \mbox{if $\mu-\alpha_i\in 
\Psi$},\\ \bar{z}_\mu & \quad \mbox{otherwise}, \end{array}\right. 
\end{align*}
where $\mu\in \Psi$. Then $x_i(t),y_i(t)\in \GL(\bar{M})$ and we obtain the
group
\[G=G_R(\Psi):=\langle x_i(t),y_i(t)\mid i\in I,t\in R \rangle\subseteq 
\GL(\bar{M}).\]
If $R=k$ is a field, then $G$ is the Chevalley group of type $\Psi$ over $k$
(as in \cite{St}). 
\end{rem}

\begin{table}[htbp]
\caption{Orbits of minuscule weights} \label{Morbits}
\begin{center}
$\renewcommand{\arraystretch}{1.2} \begin{array}{ccl} \hline \mbox{Type} 
& [\Lambda:\Z\Phi]& \mbox{Size of orbit of minuscule $\varpi_{i_0}$}\\ \hline
A_n & n+1& \binom{n+1}{i_0} \;\; (1\leq i_0 \leq n)\\ B_n & 2 & 2^n\;\; 
(i_0=1) \\ C_n & 2 & 2n\;\; (i_0=n) \\ D_n & 4 & 2^{n-1}\;\; (i_0=1,2), 
\quad 2n\;\; (i_0=n) \\ E_6 & 3 & 27 \;\; (i_0=1,6) \\ E_7 & 2 & 56 \;\; 
(i_0=7) \end{array}$
\end{center}
\end{table}

\begin{rem} \label{algo1} Let $\Psi$ be a union of orbits of minuscule
weights as above. Of particular importance is the case where $\Lambda=
\Z\Psi$; in this case, we say that $G$ is of {\em simply-connected type}. 
For each of the Dynkin diagrams in Table~\ref{Mcoxgraphs}, the constructions 
in this paper yield such a group of simply-connected type. (For the remaining 
diagrams $G_2,F_4,E_8$, there is no distinction between simply-connected 
and adjoint types and so Lusztig's construction \cite{L5} is sufficient.) 
Indeed, recall from Lemma~\ref{rem00a} that $\Z\Phi\subseteq \Z \Psi
\subseteq \Lambda$. Now the index $[\Lambda:\Z\Phi]$ is finite and given
as in Table~\ref{Morbits}. More precisely, by \cite[VIII, \S 7, Prop.~8]{B}, 
we have that the cosets 
\[ \{\varpi_{i_0}+\Z \Phi \mid i_0\in I \mbox{ such that 
$\varpi_{i_0}$ minuscule}\} \subseteq \Lambda/\Z\Phi\]
are precisely the non-zero elements of $\Lambda/\Z\Phi$. Thus, if $\Psi$ is
the union of all orbits~$\Psi_{i_0}$ where $\varpi_{i_0}$ is minuscule, then 
$G=G_k(\Psi)$ may be regarded as a ``canonical'' realisation of the 
Chevalley group of simply-connected type $\Psi$ over a field $k$. (In 
types $A_n$, $D_n$, $E_6$, one can choose $\Psi$ more economically, see the
examples below.)
\end{rem}

\begin{exmp} \label{sln} Let $n\geq 1$ and consider the Dynkin diagram of 
type $A_n$, as in Table~\ref{Mcoxgraphs}. With the above notation,
let $i_0=1$ and $\Psi=\Psi_1$. By Table~\ref{Morbits}, we have 
$|\Psi|=n+1$. In this case, one easily checks that $\Lambda=\Z \Psi$ and
\[\Psi=\{\varpi_1,\;\varpi_1-\alpha_1,\; \varpi_1-(\alpha_1+\alpha_2),\;
\ldots,\; \varpi_1-(\alpha_1+\ldots +\alpha_n)\}.\]
The matrix of $x_i(t)$ with respect to the corresponding basis of $\bar{M}$
has entries~$1$ along the diagonal, entry $t$ at the position $(i,i+1)$, 
and entry $0$ otherwise. Similarly, the matrix of $y_i(t)$ has $1$ 
along the diagonal, $t$ at the position $(i+1,i)$, and $0$ otherwise. Thus, 
we obtain the group $G=G_k(\Psi)= \mbox{SL}_{n+1}(k)$ in its standard 
respresentation. 

If $i_0\in \{2,\ldots,n-1\}$, then we obtain groups which are
neither of adjoint nor of simply-connected type in general. 
\end{exmp}

\begin{exmp} \label{classi} Let $n \geq 2$ and consider the Dynkin diagram 
of type $B_n$. Let $i_0=1$ and $\Psi=\Psi_1$. By Table~\ref{Morbits} and 
Remark~\ref{algo1}, $|\Psi|=2^n$ and $\Lambda=\Z \Psi$. In this case, our 
$\fg$-module $M$ in Definition~\ref{defj} is the so-called {\it spin} 
representation; see \cite[VIII, \S 13, p.~197]{B} and note that $M$ has 
the correct heighest weight by Remark~\ref{remhw}. Thus, $G_k(\Psi)\cong 
\mbox{Spin}_{2n+1}(k)$ is the odd-dimensional spin group. 
\end{exmp}

\begin{exmp} \label{classi1} Let $n \geq 2$ and consider the Dynkin diagram 
of type $C_n$. Let $i_0=n$ and $\Psi=\Psi_n$. By Table~\ref{Morbits} and
Remark~\ref{algo1}, we have $|\Psi|=2n$ and $\Lambda=\Z \Psi$. 
In this case, our $\fg$-module $M$ in Definition~\ref{defj} is the 
``natural'' module (equipped with a non-degenerate invariant symplectic
form); see \cite[VIII, \S 13, p.~202]{B} and note that $M$ has the correct 
heighest weight by Remark~\ref{remhw}. Hence, we find that $G_k(\Psi)
\cong\mbox{Sp}_{2n}(k)$ is the symplectic group. 
\end{exmp}

\begin{exmp} \label{classo} Let $n \geq 3$ and consider the Dynkin diagram 
of type $D_n$. Let $i_0=n$ and $\Psi=\{w(\varpi_n)\mid w\in W\}$. By 
Table~\ref{Morbits}, we have 
$|\Psi|=2n$. In this case, one easily checks that $[\Lambda:\Z \Psi]=2$.
Our $\fg$-module $M$ in Definition~\ref{defj} is the ``natural'' module
(equipped with a non-degenerate invariant quadratic form); see 
\cite[VIII, \S 13, p.~209]{B} and note that $M$ has the correct heighest 
weight by Remark~\ref{remhw}. Thus, $G_k(\Psi)\cong \mbox{SO}_{2n}(k)$ is 
the even-dimensional orthogonal group. (It is neither of adjoint nor of 
simply-connected type.)
\end{exmp}

\begin{exmp} \label{spin} Let $n \geq 3$ and consider the Dynkin diagram
of type $D_n$ in Table~\ref{Mcoxgraphs}. Let $\Psi:=\Psi_1\cup \Psi_2$
(the union of the orbits of $\varpi_1$ and $\varpi_2$). By 
Table~\ref{Morbits} and Remark~\ref{algo1}, $|\Psi|=2^{n-1}+2^{n-1}=2^n$ and
$\Lambda=\Z\Psi$. In this case, our $\fg$-module $M$ in Definition~\ref{defj} 
is the direct sum of the two so-called {\it half-spin} representations; see 
\cite[VIII, \S 13, p.~209]{B} and note again that $M$ has the correct
heighest weights by Remark~\ref{remhw}. Hence, we conclude that $G=
G_k(\Psi) \cong \mbox{Spin}_{2n}(k)$ is the spin group.

This group plays a special role in the general theory because of its center. 
Let us explicitly determine $Z(G)$. We write $h(t_1,\ldots,t_n):=h_1(t_1) 
\cdots h_n(t_n)\in H$ for $t_i\in k^\times$. Then every $h\in H$ can be 
expressed uniquely in this way (since $\Z\Psi=\Lambda$). A straightforward 
computation yields that 
\[ Z(G)=\left\{\begin{array}{ll} \{h(t,t',1,tt',1,tt',1,\ldots) \mid t^2=
t'^2=1\} &\mbox{ if $n$ is even},\\ \;\;\{h(t,t^{-1},t^2,1,t^2,1,\ldots) 
\mid t^4=1\} & \mbox{ if $n$ is odd}.\end{array}\right.\]
Thus, if $n$ is even and $\mbox{char}(k)\neq 2$, then $Z(G)\cong \Z/2\Z 
\times \Z/2\Z$; in all other cases and for all other types of groups, 
$Z(G)$ is cyclic.
\end{exmp}

\begin{exmp}  \label{sce6} (a) Consider the Dynkin diagram of type $E_6$
in Table~\ref{Mcoxgraphs}. Let $i_0=1$ and $\Psi=\Psi_1$. By 
Table~\ref{Morbits} and Remark~\ref{algo1}, $|\Psi|=27$ and $\Lambda=\Z 
\Psi$. In this case, $\Psi$ is explicitly given by the following 
$6$-tuples (where $\pm$ stands for $\pm 1$):
\begin{gather*}
 {+}00000,\; {-}0{+}000,\; 00{-}{+}00,\; 0{+}0{-}{+}0,\; 0{-}00{+}0,\; 
0{+}00{-}{+},\; 0{-}0{+}{-}{+},\;\\ 0{+}000{-},\; 00{+}{-}0{+},\; 
0{-}0{+}0{-},\; {+}0{-}00{+},\; 00{+}{-}{+}{-},\; {-}0000{+},\; 
{+}0{-}0{+}{-},\; \\ 00{+}0{-}0,\; {-}000{+}{-},\; {+}0{-}{+}{-}0,\; 
{-}00{+}{-}0,\; {+}{+}0{-}00,\; {-}{+}{+}{-}00,\; {+}{-}0000,\; \\ 
{-}{-}{+}000,\; 0{+}{-}000,\; 0{-}{-}{+}00,\; 000{-}{+}0,\; 0000{-}{+},\; 
00000{-}0
\end{gather*}
(b) Consider the Dynkin diagram of type $E_7$. Let $i_0=7$ and $\Psi=\Psi_7$.
By Table~\ref{Morbits} and Remark~\ref{algo1}, $|\Psi|=56$ and $\Lambda=\Z 
\Psi$. The set $\Psi$ can also be computed as above, but we will not print
this here.
\end{exmp}

For other approaches to constructions related to minuscule representations
and the question of signs, see Green \cite{Gr}, Vavilov 
\cite{Va} and the references there. 


\end{document}